%% file: maxleg.tex
\documentclass[11pt]{amsart}

\usepackage[T1]{fontenc}
\usepackage[utf8]{inputenc}
\usepackage{lmodern}

\usepackage[american]{babel}
\usepackage[babel, final]{microtype}
\usepackage{amssymb}

\usepackage[pdftex,%
  dvipsnames,%
]{xcolor}
\usepackage[pdftex, final]{graphicx}
\usepackage{pinlabel}
\usepackage[pdftex,%
  letterpaper,%
  includehead,%
  includefoot,%
  nomarginpar,%
  lmargin=1in,%
  rmargin=1in,%
  tmargin=1in,%
  bmargin=1in,%
]{geometry}
\usepackage[pdftex,%
  final,%
  colorlinks=true,%
  linkcolor=NavyBlue,%
  citecolor=NavyBlue,%
  filecolor=NavyBlue,%
  menucolor=NavyBlue,%
  urlcolor=NavyBlue,%
  bookmarks=true,%
  bookmarksdepth=3,%
  bookmarksnumbered=true,%
  bookmarksopen=true,%
  bookmarksopenlevel=2,%
]{hyperref}
\hypersetup{
  pdfinfo={
    Title={Upper bounds for the Lagrangian cobordism relation on Legendrian 
      links},
    Author={Joshua M. Sabloff, David Shea Vela-Vick, C.-M. Michael Wong},
    Subject={Contact geometry, knot theory, Lagrangian cobordism},
    Keywords={Legendrian links, Lagrangian cobordisms}
  }
}

\input{macros}

\input{commands}

\input{commands_local}

\title[Upper bounds for the Lagrangian cobordism relation]{Upper bounds for the 
  Lagrangian cobordism relation on Legendrian links}

\author[Joshua M. Sabloff]{Joshua M. Sabloff}
\address{Department of Mathematics and Statistics \\ Haverford College \\  
  Haverford, PA 19041}
\email{\href{mailto:jsabloff@haverford.edu}{jsabloff@haverford.edu}}
\urladdr{\url{http://ww3.haverford.edu/math/jsabloff/}}

\author[David Shea Vela-Vick]{David Shea Vela-Vick}
\address{Department of Mathematics \\ Louisiana State University \\ Baton 
  Rouge, LA 70803}
\email{\href{mailto:shea@math.lsu.edu}{shea@math.lsu.edu}}
\urladdr{\url{http://www.math.lsu.edu/~shea/}}
\thanks{DSV was partially supported by NSF Grant DMS-1907654 and Simons 
  Foundation Grant 524876.}

\author[C.-M. Michael Wong]{C.-M. Michael Wong}
\address{Department of Mathematics \\ Dartmouth College \\ Hanover, NH 03755}
\email{\href{mailto:wong@math.dartmouth.edu}{wong@math.dartmouth.edu}}
\urladdr{\url{https://math.dartmouth.edu/~wong/}}
\thanks{CMMW was partially supported by NSF Grant DMS-2039688 and an AMS-Simons 
  Travel Grant.}


\begin{document}


\begin{abstract}

  Lagrangian cobordism induces a preorder on the set of Legendrian links in any 
  contact $3$-manifold. We show that any finite collection  of null-homologous 
  Legendrian links in a tight contact $3$-manifold with a common rotation 
  number has an upper bound with respect to the preorder. In particular, we 
  construct an exact Lagrangian cobordism from each element of the collection 
  to a common Legendrian link.  This construction allows us to define a notion 
  of minimal Lagrangian genus between any two null-homologous Legendrian links 
  with a common rotation number. 

\end{abstract}


\maketitle


\input{sec_intro}

\input{sec_desc}

\input{sec_min}

\input{sec_min_diag}

\input{sec_max}

\input{sec_genus}


\bibliographystyle{mwamsalphack}
\bibliography{bibliography}


\end{document}

%% file: commands.tex
\newcommand{\set}[1]{\mathchoice%
  {\left\lbrace #1 \right\rbrace}%
  {\lbrace #1 \rbrace}%
  {\lbrace #1 \rbrace}%
  {\lbrace #1 \rbrace}%
}

\newcommand{\union}{\cup}
\newcommand{\bigunion}{\bigcup}
\newcommand{\intersect}{\cap}

\newcommand{\comp}{\circ}
\newcommand{\cross}{\times}

\newcommand{\numset}[1]{\mathbb{#1}}

\newcommand{\R}{\numset{R}}

\newcommand{\bdy}{\partial}

\newcommand{\unknot}{\mathord{\bigcirc}}

\newcommand{\xistd}{\xi_{\mathrm{std}}}
\DeclareMathOperator{\tb}{tb}
\DeclareMathOperator{\rot}{r}

%% file: commands_local.tex
\newcommand{\disjunion}{\sqcup}

\newcommand*{\leg}{\Lambda}
\newcommand*{\legp}{\leg_+}
\newcommand*{\legm}{\leg_-}
\newcommand*{\legpm}{\leg_\pm}

\newcommand*{\legone}{\leg}
\newcommand*{\legtwo}{\leg'}

\newcommand*{\link}{\leg}

\newcommand*{\linkm}{\link_-}

\newcommand*{\linkone}{\link}
\newcommand*{\linktwo}{\link'}

\newcommand*{\seif}{\Sigma}

\newcommand*{\seifone}{\seif}
\newcommand*{\seiftwo}{\seif'}

\newcommand*{\cob}{L}

\newcommand*{\cobone}{\cob}
\newcommand*{\cobtwo}{\cob'}

\newcommand*{\graph}{G}

\newcommand*{\graphone}{\graph}
\newcommand*{\graphtwo}{\graph'}

\newcommand*{\graphoneb}{\overline{\graphone}}
\newcommand*{\unknotoneb}{\overline{\Lambda}_U}
\newcommand*{\graphonet}{\widetilde{\graphone}}

\newcommand*{\handles}{\mathcal{H}}

\newcommand*{\handlesone}{\handles}
\newcommand*{\handlestwo}{\handles'}

\newcommand*{\divset}{\Gamma}

\newcommand*{\cobrel}{\preceq}

\newcommand*{\arc}{\gamma}
\newcommand*{\disk}{D}

\DeclareMathOperator{\Surg}{Surg}

\DeclareMathOperator{\tw}{tw}

\newcommand*{\order}{\mathfrak{o}}

\newcommand*{\relgen}{g_{\cob}}

\newcommand*{\maxunknot}{\Upsilon}
\renewcommand*{\unknot}{\leg_U}

\newcommand*{\unknotone}{\unknot}
\newcommand*{\unknottwo}{\unknot'}

\renewcommand*{\xistd}{\xi_0}
\newcommand*{\alphastd}{\alpha_0}

\newcommand*{\diskstd}{\disk_0}

\newcommand*{\linkstd}{\link_0}

\newcommand*{\pstab}{S_+}
\newcommand*{\nstab}{S_-}

\newcommand*{\strand}{\eta}

\newcommand*{\smpath}{\gamma}

\newcommand*{\legsint}{\boldsymbol{\leg}}
\newcommand*{\legsuplow}{\boldsymbol{\leg}^*}
\newcommand*{\cobsint}{\mathbf{L}}

%% file: sec_intro.tex
\section{Introduction}
\label{sec:intro}

The relation $\cobrel$ defined by (exact, orientable) Lagrangian cobordism 
between Legendrian submanifolds in the symplectization of the contact manifold 
raises a host of surprisingly subtle structural questions. While the Lagrangian 
cobordism relation is trivially a preorder (i.e.\ is reflexive and transitive), 
it is not symmetric \cite{BalSiv18:KHMLeg, Cha10:LagConc, 
  CorNgSiv16:LagConcObstructions}; it is unknown whether the relation is a 
partial order.  Further, not every pair of Legendrians is related by Lagrangian 
cobordism, with the first obstructions coming from the classical invariants:  
For links $\legpm$ in $\R^3$, if $\legm \cobrel \legp$ via the Lagrangian $\cob 
\subset \R \cross \R^3$, then $\rot(\legp) = \rot(\legm)$ and $\tb(\legp) - \tb 
(\legm) = - \chi (\cob)$ \cite{Cha10:LagConc}.  A growing toolbox of 
non-classical obstructions has been developed to detect this phenomenon; see, 
just to begin, \cite{BalSiv18:KHMLeg, BalLidWon21:LagCobHFK,  
  EkhHonKal16:LagCob, GolJuh19:LOSSConc, Pan17:LagCobAug, 
  SabTra13:LagCobObstructions}.  

If two Legendrians are not related by a Lagrangian cobordism,  one may still 
ask if they have a common upper or lower bound with respect to $\cobrel$.  
Implicit in the work of Boranda, Traynor, and Yan \cite{BorTraYan13:MinLeg} is 
that any finite collection of Legendrian links in the standard contact $\R^3$ 
with the same rotation number has a lower bound with respect to $\cobrel$.  In 
another direction, Lazarev \cite{Laz20:MaxContSymp} has shown that any finite 
collection of formally isotopic Legendrians in a contact $(2n+1)$-manifold with 
$n \geq 2$ has an upper bound with respect to a moderate generalization of 
$\cobrel$.

The goal of this paper is to find both lower and upper bounds for finite 
collections of Legendrian links in any tight contact $3$-manifold.  On one 
hand, in contrast to the diagrammatic methods of \cite{BorTraYan13:MinLeg}, our 
topological techniques allow us to find lower bounds in any tight contact 
$3$-manifold, though we also present a refinement of the proof in 
\cite{BorTraYan13:MinLeg} that better suits our goal of constructing upper 
bounds.  On the other hand, in contrast to Lazarev's use of an $h$-principle, 
which restricts his results to higher dimensions, our direct constructions of 
upper bounds work for Legendrian links in dimension $3$.

\begin{theorem}
  \label{thm:main}
  Let $\legone$ and $\legtwo$ be oriented Legendrian links in a tight contact 
  $3$-manifold $(Y, \xi)$, and suppose that there exist Seifert surfaces 
  $\seifone$ and $\seiftwo$ for which $\rot_{[\seifone]} (\legone) = 
  \rot_{[\seiftwo]} (\legtwo)$.   Then there exist oriented Legendrian links 
  $\leg_\pm \subset (Y, \xi)$ such that $\legm \cobrel \legone \cobrel \legp$ 
  and $\legm \cobrel \legtwo \cobrel \legp$.
\end{theorem}

\begin{remark}
  For Legendrian links in $\R^3$, the rotation number may be defined without 
  reference to Seifert surfaces, and the hypotheses merely require 
  $\rot(\legone) = \rot(\legtwo)$.
\end{remark}

\begin{example}
  In \fullref{fig:main-example1}, we display an upper bound for the maximal 
  Legendrian right-handed trefoil and a Legendrian $m(5_2)$ knot.  These two 
  Legendrian knots are not related by Lagrangian cobordism.  To see why, note 
  that any Lagrangian cobordism between them must be a concordance since they 
  have the same Thurston--Bennequin number, but no such concordance exists even 
  topologically.

  \begin{figure}[htbp]
    \includegraphics[width=4in]{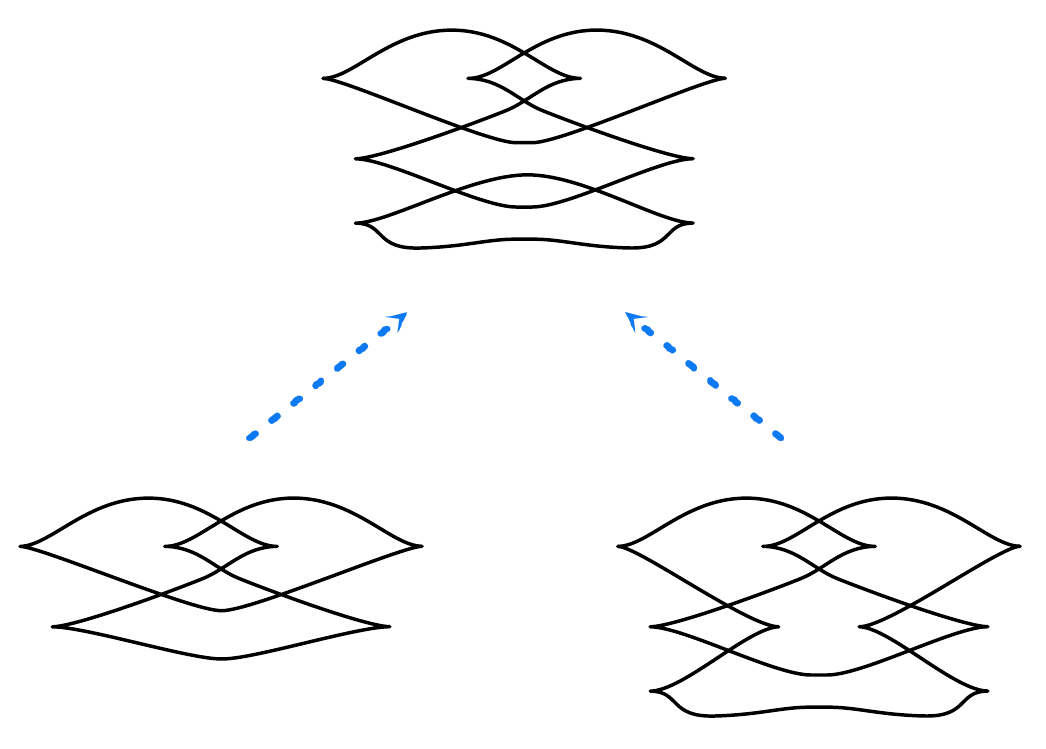}
    \caption{An upper bound for the maximal right-handed trefoil and an 
      $m(5_2)$ knot.}
    \label{fig:main-example1}
  \end{figure}
\end{example} 

\begin{example}
  In \fullref{fig:main-example2}, we display an upper bound for the maximal 
  Legendrian unknot and the maximal Legendrian figure-eight knot.  Once again, 
  these two Legendrian knots are not related by Lagrangian cobordism.  The fact 
  that the figure-eight has lower Thurston--Bennequin number shows that there 
  cannot be a cobordism from the unknot to the figure-eight; the fact that the 
  figure-eight has two normal rulings shows that there cannot be a cobordism 
  from the figure-eight to the unknot \cite[Theorem 
  2.7]{CorNgSiv16:LagConcObstructions}.

  \begin{figure}[htbp]
    \includegraphics[width=4in]{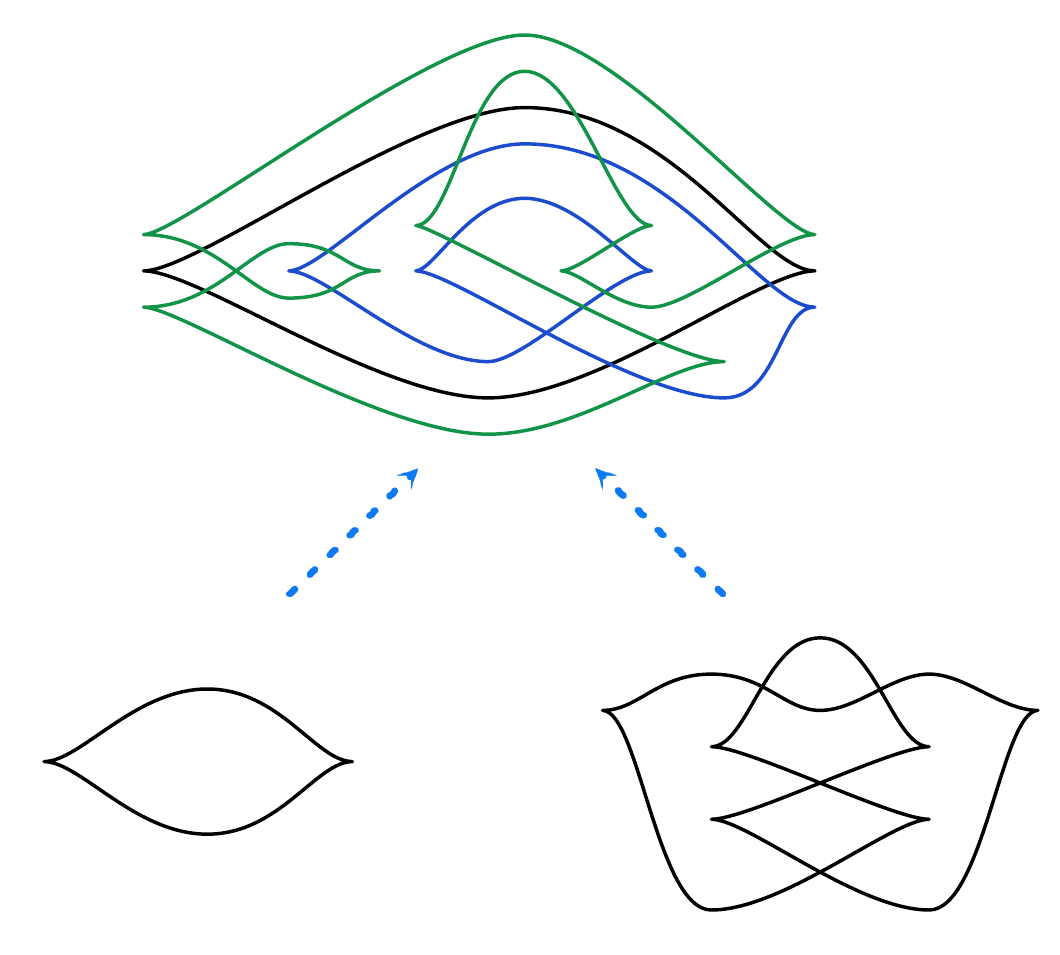}
    \caption{An upper bound for the maximal unknot and the maximal figure-eight 
      knot.  The colors in the diagram of the upper bound are only meant to 
      distinguish components of the link to improve readability.}
    \label{fig:main-example2}
  \end{figure}
\end{example}

In fact, we prove the following strengthened version of \fullref{thm:main}.

\begin{proposition}
  \label{prop:slices}
  Under the same hypotheses of \fullref{thm:main}, there exist oriented 
  Legendrian links $\legm, \legp \subset Y$ and oriented exact decomposable 
  Lagrangian cobordisms $\cobone$ and $\cobtwo$ from $\legm$ to $\legp$, such 
  that
  \begin{itemize}
    \item The Legendrian link $\legone$ appears as a collared slice of 
      $\cobone$;
    \item The Legendrian link $\legtwo$ appears as a collared slice of 
      $\cobtwo$; and
    \item $\cobone$ and $\cobtwo$ are exact-Lagrangian isotopic.
  \end{itemize}
\end{proposition}

\begin{remark}
  There are statements analogous to \fullref{thm:main} and 
  \fullref{prop:slices} that hold for unoriented Legendrian links and 
  unoriented (and possibly unorientable) exact Lagrangian cobordisms, for which 
  there are no requirements on the rotation number.
\end{remark}

The main theorem has several interesting consequences.  First, we recall that 
not every Legendrian knot has a Lagrangian filling.  The figure-eight knot in 
\fullref{fig:main-example2} is one such example.  By transitivity, this implies 
that not every Legendrian knot lies at the top of a Lagrangian cobordism from a 
fillable Legendrian.  On the other hand, we have the following corollary of the 
main theorem:

\begin{corollary}
  \label{cor:cob-to-fillable}
  For any Legendrian link $\leg$, there exists a Legendrian link $\legp$ with a 
  Lagrangian filling and a Lagrangian cobordism from $\leg$ to $\legp$.
\end{corollary}

The proof simply requires us to apply \fullref{thm:main} with $\legone$ being 
the given Legendrian and $\legtwo$ being the maximal Legendrian unknot.  The 
upper bound $\legp$ is Lagrangian fillable since there is a cobordism to it 
from the unknot.

A second consequence of the main theorem is that we are able to define a notion 
of the minimal genus of a Lagrangian cobordism between \emph{any} two 
Legendrian links with the same rotation number.  Roughly speaking, we define a 
Lagrangian quasi-cobordism between $\legone$ and $\legtwo$ to be a sequence 
$\legone = \leg_0, \leg_1, \cdots, \leg_n = \legtwo$ of Legendrian links 
together with upper (or lower) bounds between each of $\leg_i$ and 
$\leg_{i+1}$.  The genus of the quasi-cobordism is the genus of the (smooth) 
composition of the underlying Lagrangian cobordisms between the $\leg_i$ and 
their bounds; we may then define $g_L(\legone, \legtwo)$ to be the minimal 
genus of such a Lagrangian quasi-cobordism. When there is a Lagrangian 
cobordism from $\legone$ to $\legtwo$ and $\legone$ is fillable, $\relgen 
(\legone, \legtwo)$ agrees with the relative smooth genus $g_s (\legone, 
\legtwo)$; see \fullref{lem:genus-comparison}.

The remainder of the paper is organized as follows:  In 
\fullref{sec:description}, we review key ideas in the definition and 
construction of Lagrangian cobordisms between Legendrian links.  We also define 
the notion of a Legendrian handle graph, which will form the basis of our later 
constructions.  In \fullref{sec:min} and \fullref{sec:min_diag}, we prove that 
any two Legendrians in a tight contact $3$-manifold have a lower bound with 
respect to $\cobrel$, and encode the Lagrangian cobordisms involved with 
Legendrian handle graphs. We present two approaches to this goal: In 
\fullref{sec:min}, we prove the claim for general tight contact $3$-manifolds 
using convex surface theory, while in \fullref{sec:min_diag}, we provide a 
diagrammatic proof in $\R^3$, refining a proof of \cite{BorTraYan13:MinLeg}. We 
then proceed in \fullref{sec:max} to prove \fullref{prop:slices}, and hence 
\fullref{thm:main}.  We end the paper in \fullref{sec:genus} by beginning an 
exploration of Lagrangian quasi-cobordisms and their genera, finishing with 
some open questions.

\subsection*{Acknowledgments}
The authors thank Oleg Lazarev for discussions of his work 
\cite{Laz20:MaxContSymp} that motivated this project and for further dialogue 
once this project began in earnest. Part of the research was conducted while 
the third author was at Louisiana State University. The first author thanks 
Louisiana State University, and the third author thanks Haverford College, for 
their hospitality.

%% file: sec_desc.tex
\section{A description of Lagrangian cobordisms}
\label{sec:description}

In this section, we describe Lagrangian cobordisms, how to construct them, and 
how to keep track of those constructions.

\subsection{Lagrangian cobordisms}
\label{ssec:cobordism}

We begin with the formal definition of a Lagrangian cobordism between 
Legendrian links.

\begin{definition}
  \label{def:cobordism}
  Let $\leg_-$ and $\leg_+$ be Legendrian links in the contact manifold $(Y, 
  \xi)$, where $\xi = \ker (\alpha)$ for a contact $1$-form $\alpha$. An 
  (exact, orientable) \emph{Lagrangian cobordism} $\cob$ from $\leg_-$ to 
  $\leg_+$ is an exact, orientable, properly embedded Lagrangian submanifold 
  $\cob \subset (\R \times Y, d (e^t \alpha))$ that satisfies the following:
  \begin{itemize}
    \item There exists $T_+ \in \R$ such that $\cob \cap ([T_+,\infty) \times 
      Y) = [T_+, \infty) \times \leg_+$;
    \item There exists $T_- < T_+$ such that $\cob \cap ((-\infty, T_-] \times 
      Y) = (-\infty, T_-] \times \leg_-$; and
    \item The primitive of $(e^t \alpha) \vert_\cob$ is constant (rather than 
      locally constant) at each cylindrical end of $\cob$.
  \end{itemize}
\end{definition}

Note that the last condition enables us to concatenate Lagrangian cobordisms 
while preserving exactness.

We will use three constructions of Lagrangian cobordisms in this paper, which 
we will call the \emph{elementary Lagrangian cobordisms}:
\begin{description}
  \item[$0$-handle]  Adding a disjoint, unlinked maximal Legendrian unknot 
    $\maxunknot$ to $\link$ induces an exact Lagrangian cobordism from $\link$ 
    to $\link \disjunion \maxunknot$ \cite{BouSabTra15:LagCobGF, 
      EkhHonKal16:LagCob}.
  \item[Legendrian isotopy] A Legendrian isotopy from $\link$ to $\link'$ 
    induces an exact Lagrangian cobordism from $\link$ to $\link'$, though the 
    construction is more complicated than simply taking the trace of the 
    isotopy \cite{BouSabTra15:LagCobGF, 
      EkhHonKal16:LagCob,EliGro98:LagrIntThy}.
  \item[Legendrian ambient surgery] We describe this construction in more 
    detail in \fullref{ssec:rizell}, and we will develop a method for keeping 
    track of a set of ambient surgeries in \fullref{ssec:graph}.
\end{description}

\subsection{Legendrian ambient surgery}
\label{ssec:rizell}

Our next step is to explain Dimitroglou Rizell's Legendrian ambient surgery 
construction in the $3$-dimensional setting \cite{Dim16:LegAmbSurg}.  Similar 
constructions appear in \cite{BouSabTra15:LagCobGF} and 
\cite{EkhHonKal16:LagCob}, though Dimitroglou Rizell's more flexible language 
is best suited for our purposes.  In dimension $3$, Legendrian ambient surgery 
begins with the data of an oriented Legendrian link $\link \subset (Y, \xi)$ 
and an embedded Legendrian curve $\disk$ with endpoints on $\link$ that is, in 
a sense to be defined, compatible with the orientation of $\link$.  The 
construction then produces a Legendrian $\link_\disk$, contained in an 
arbitrarily small neighborhood of $\link \cup \disk$, that is obtained from 
$\link$ by ambient surgery along $\disk$.  Further, the construction produces 
an exact Lagrangian cobordism from $\link$ to $\link_\disk$.  

More precisely, given $\leg \subset (Y, \xi)$ with contact $1$-form $\alpha$, a 
\emph{surgery disk} is an embedded Legendrian arc $\disk \subset Y$ such that
\begin{enumerate}
  \item $\disk \cap \link = \partial \disk$,
  \item The intersection $\disk \cap \link$ is transverse, and
  \item The vector field $H \subset T_p \leg$ defined for all $p \in \bdy 
    \disk$ (up to scaling) by $d \alpha (G, H (p)) > 0$ for all 
    outward-pointing vectors $G$ in $T_p \disk$ either completely agrees with 
    or completely disagrees with the framing on $\partial \disk$ induced by the 
    orientation of $\leg$.
\end{enumerate}

For an unoriented surgery, we need not specify a framing for $\partial \disk$, 
and the last condition is no longer relevant.

The standard model for such a surgery disk appears in 
\fullref{fig:standard-disk}~(a).  In fact, up to an overall orientation 
reversal on $\link$, there is a neighborhood $U$ of $\disk$ in $Y$ that is 
contactomorphic to a neighborhood of the standard model for $\leg_0$ and 
$\diskstd$ \cite[Section~4.4.1]{Dim16:LegAmbSurg}.  Working in the standard 
model, we may replace $\linkstd$ by the Legendrian arcs $\link_1$ as in 
\fullref{fig:standard-surgery}, a process that realizes the ambient surgery on 
$\linkstd$ along $\diskstd$.  Pulling this construction back to the 
neighborhood of $\disk$ in $Y$, we call the resulting link \emph{Legendrian 
  ambient surgery} on $\link$ along $\disk$.  

\begin{figure}[htbp]
  \includegraphics{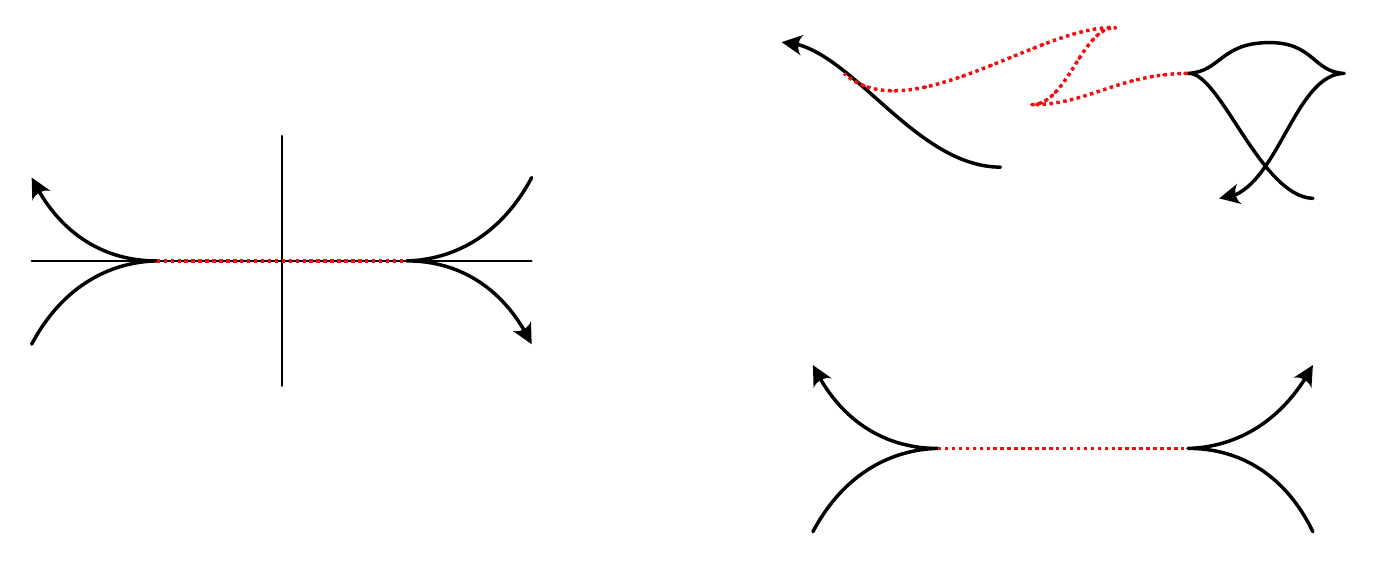}
  \caption{(a) The standard model in $(\R^3, \alphastd)$ of a surgery disk 
    $\diskstd$ with endpoints on a Legendrian $\linkstd$, with (b) another 
    example of a surgery disk, and (c) a disk that fails condition (3).}
  \label{fig:standard-disk}
\end{figure}

\begin{figure}[htbp]
  \includegraphics{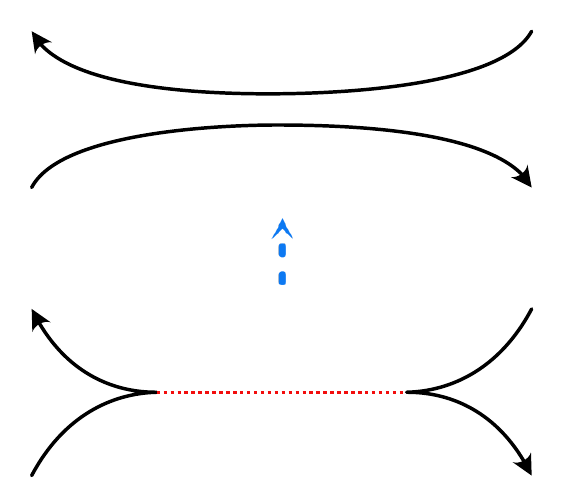}
  \caption{Surgery on the standard model $\link_0 \cup D$ yields a new 
    Legendrian $\link_1$.}
  \label{fig:standard-surgery}
\end{figure}

\begin{theorem}[Dimitroglou Rizell \cite{Dim16:LegAmbSurg}]
  \label{thm:rizell}
  Given an oriented Legendrian link $\link$ and a surgery disk $\disk$, let 
  $\link_\disk$ be the Legendrian link obtained from $\link$ by Legendrian 
  ambient surgery along $\disk$.  Then there exists an exact Lagrangian 
  cobordism from $\link$ to $\link_\disk$ arising from the attachment of a 
  $1$-handle to $(-\infty,T] \times \link$.
\end{theorem}

\begin{remark}
  \label{rmk:local-surgery}
  The construction of Legendrian ambient surgery and the associated Lagrangian 
  cobordism is local.  In particular, for a small neighborhood $U$ of $\disk$, 
  the surgery construction does not alter $\leg \cap (Y \setminus U)$, and the 
  cobordism $\cob$ outside of $\R \times U$ is cylindrical over $\leg \cap (Y 
  \setminus U)$.
\end{remark}

\subsection{Legendrian handle graphs}
\label{ssec:graph}

In this section, we introduce a structure for keeping track of independent 
ambient surgeries.  We use the notion of a Legendrian graph, following the 
conventions in \cite{ODoPav12:LegGraphs}.

Before we begin, recall from e.g.\ \cite{ODoPav12:LegGraphs} that two 
Legendrian graphs in $(\R^3, \xistd)$ are Legendrian isotopic if and only if 
their front diagrams are related by planar isotopy and six Reidemeister moves, 
as seen in \fullref{fig:reidemeister}.

\begin{figure}[htbp]
  \labellist
  \tiny\hair 2pt
  \pinlabel {I} [ ] at 154 295
  \pinlabel {II} [ ] at 154 182
  \pinlabel {III} [ ] at 154 67
  \pinlabel {IV} [ ] at 522 313
  \pinlabel {IV} [ ] at 655 313
  \pinlabel {V} [ ] at 582 182
  \pinlabel {VI} [ ] at 582 67
  \endlabellist
  \includegraphics[width=5in]{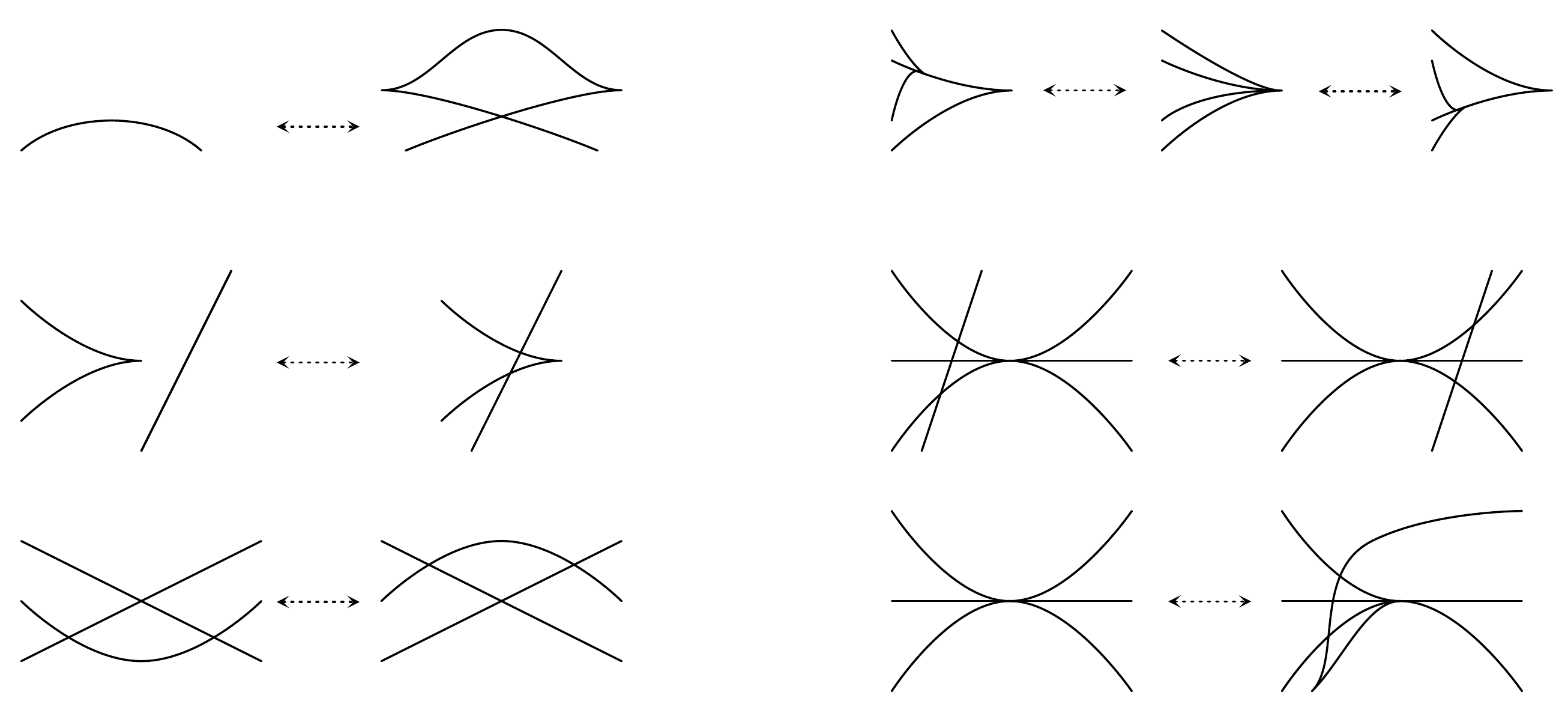}
  \caption{Reidemeister moves for Legendrian graphs in $\R^3$.}
  \label{fig:reidemeister}
\end{figure}
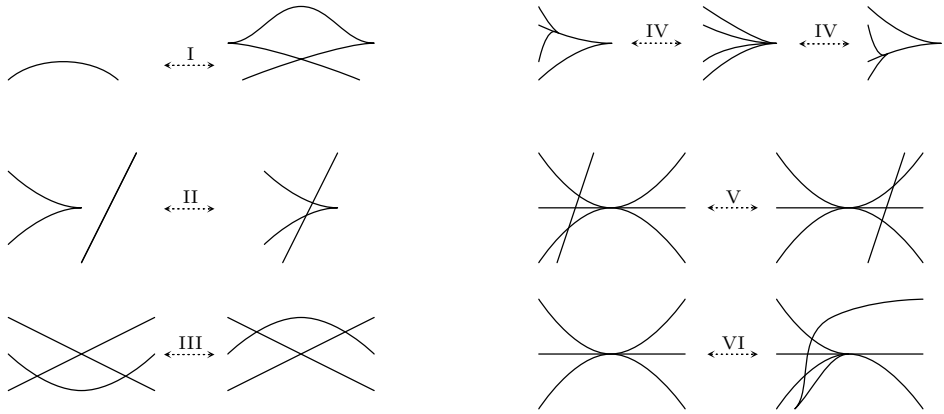

\begin{definition}
  \label{def:graph}
  A \emph{Legendrian handle graph} is a pair $(\graph, \leg)$, where $\graph 
  \subset (Y, \xi)$ is a trivalent Legendrian graph and $\leg \subset (Y, \xi)$ 
  is a Legendrian link (called the \emph{underlying link}), such that
  \begin{itemize}
    \item $\leg \subset \graph$;
    \item The vertices of $\graph$ lie on $\leg$; and
    \item $\graph \setminus \leg$ is the union of a finite collection of 
      pairwise disjoint Legendrian arcs $\arc_1, \dotsc, \arc_m$ whose closures 
      satisfy the conditions of surgery disks for $\leg$.
  \end{itemize}
  We also say that $\graph$ is a \emph{Legendrian handle graph on $\leg$}. The 
  set of closures of the components of $\graph \setminus \leg$ is denoted by 
  $\handles$.
\end{definition}

See the bottom of \fullref{fig:handle-graph} for an example of a Legendrian 
handle graph whose underlying Legendrian link is a Legendrian Hopf link in 
$(\R^3, \xistd)$.

\begin{figure}[htbp]
  \includegraphics{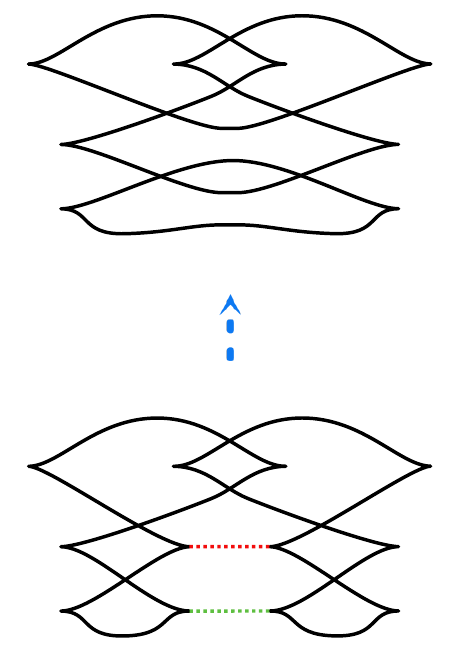}
  \caption{The Legendrian link at the top of the figure is the Legendrian 
    ambient surgery on the Legendrian handle graph $(\graph, \leg)$ at the 
    bottom.}
  \label{fig:handle-graph}
\end{figure}

\begin{definition}
  \label{def:graph-surg}
  Let $(\graph, \leg)$ be a Legendrian handle graph and let $\handles_0$ be a 
  subset of the arcs in $\handles$. The \emph{Legendrian ambient surgery} 
  $\Surg (\graph, \leg, \handles_0)$ is the Legendrian handle graph $(\graph', 
  \leg')$ resulting from performing Legendrian ambient surgery along each arc 
  in $\handles_0$, as described in \fullref{ssec:rizell}.
\end{definition}

We will, at times, abuse notation and refer to the underlying Legendrian link 
$\leg'$ by $\Surg (\graph, \leg, \handles_0)$; we will also use $\Surg 
(\graph,\leg)$ when $\handles_0 = \handles$. For example, in 
\fullref{fig:handle-graph}, the Legendrian link at the top is $\Surg (\graph, 
\leg)$ for the Legendrian handle graph $(\graph, \leg)$ at the bottom.

By the work of Dimitroglou Rizell \cite{Dim16:LegAmbSurg} as described in 
\fullref{ssec:rizell}, Legendrian ambient surgery on any given Legendrian arc 
corresponds to an exact Lagrangian  cobordism; this implies that, given an 
order $\order = (\arc_{j_1}, \dotsc, \arc_{j_m})$ of the components of 
$\handles_0$, one obtains an exact Lagrangian cobordism $\cob (\graph, \leg, 
\handles_0, \order)$ from $\leg$ to $\Surg (\graph, \leg, \handles_0)$ by 
performing Legendrian ambient surgery in the order given by $\order$.  The 
order, in fact, does not matter.

\begin{proposition}
  \label{prop:graph-surg-order}
  Suppose $(\graph, \leg)$ is a Legendrian handle graph, and $\order_1$ and 
  $\order_2$ are orders of the components of $\handles_0$. The Lagrangian 
  cobordisms $\cob (\graph, \leg, \handles_0, \order_1)$ and $\cob (\graph, 
  \leg, \handles_0, \order_2)$ are exact-Lagrangian isotopic.
\end{proposition}

\begin{proof}
  It suffices to consider the case where $\order_1$ and $\order_2$ differ by an 
  adjacent transposition $(\arc_{j_1}, \arc_{j_2}) \to (\arc_{j_2}, 
  \arc_{j_1})$.  The cobordism $\cob (\graph, \leg, \handles_0, \order)$ is 
  defined by composing the elementary Lagrangian cobordisms associated to the 
  arcs $\arc_1, \dotsc, \arc_m$. Since there are finitely many of these, by 
  shrinking the neighborhoods of $\arc_j$ as in \fullref{rmk:local-surgery}, we 
  may assume the neighborhoods to be pairwise disjoint. This implies that the 
  elementary Lagrangian cobordisms associated to $\arc_{j_1}$ and $\arc_{j_2}$ 
  may be constructed simultaneously and shifted past each other along the 
  cylindrical parts of the cobordism.  Thus, the parameter given by the 
  relative heights of these two cobordisms gives an exact Lagrangian isotopy 
  between $\cob (\graph, \leg, \handles_0, \order_1)$ and $\cob (\graph, \leg, 
  \handles_0, \order_2)$.  \end{proof}

\fullref{prop:graph-surg-order} allows us to associate an \emph{isotopy class} 
$\cob (\graph, \leg, \handles_0)$ of exact Lagrangian cobordisms to a 
Legendrian handle graph $(\graph, \leg)$ and $\handles_0 \subset \handles$. 

\begin{remark}
  \label{rmk:handle-graph-limitation}
  It is not clear that every decomposable cobordism can be described using a 
  Legendrian handle graph.  As shown in \fullref{fig:ordered-attachment}, one 
  may need to perform one ambient surgery in order for another to be possible; 
  this would violate \fullref{prop:graph-surg-order}.
\end{remark}

\begin{figure}[htbp]
  \includegraphics[width=2in]{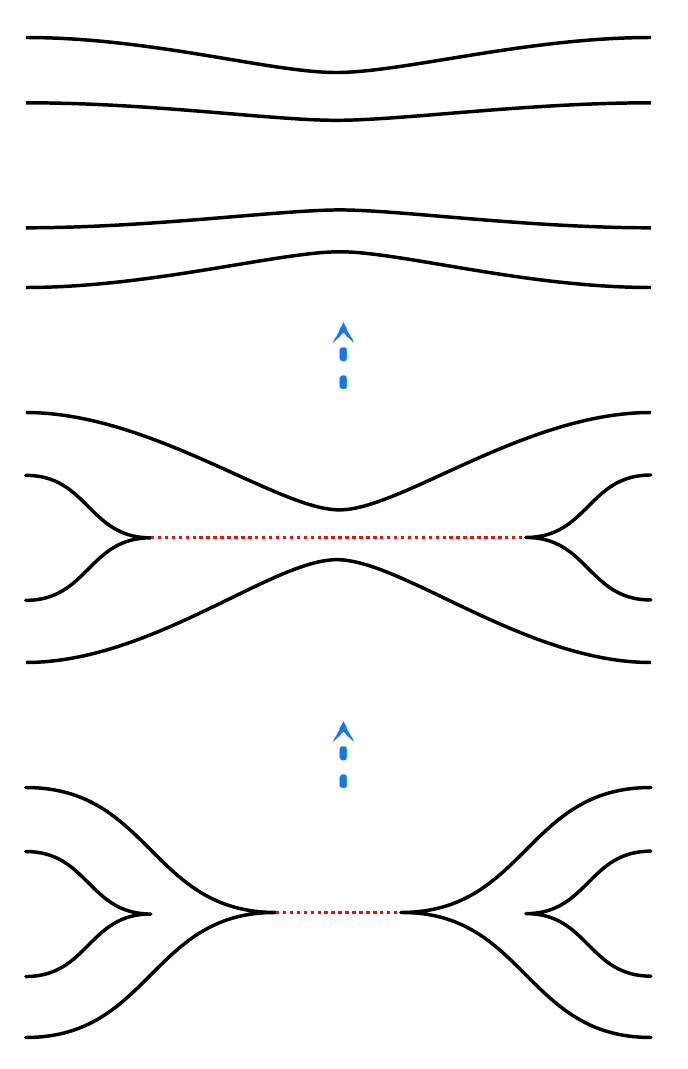}
  \caption{The surgery joining the two inner cusps cannot be performed until 
    after the surgery joining the two outer cusps.}
  \label{fig:ordered-attachment}
\end{figure}

%% file: sec_min.tex
\section{Lower bounds via contact topology}
\label{sec:min}

In this section, for a pair of Legendrian links with the same rotation number, 
we construct a pair of exact Lagrangian cobordisms from a common lower bound, 
encoded by Legendrian handle graphs with the same underlying link.

\begin{proposition}
  \label{prop:min_general}
  Let $\linkone$ and $\linktwo$ be oriented Legendrian links in a tight contact 
  manifold $(Y, \xi)$ and suppose that there exist Seifert surfaces $\seifone$ 
  and $\seiftwo$ for which $\rot_{[\seifone]} (\linkone) = \rot_{[\seiftwo]} 
  (\linktwo)$. Then there exists an oriented Legendrian link $\linkm \subset 
  (Y, \xi)$ and handle graphs $\graphone$ and $\graphtwo$ on $\leg_-$ such that 
  $\Surg(\graphone, \linkm)$ (resp.\ $\Surg (\graphtwo, \linkm)$) is Legendrian 
  isotopic to $\linkone$ (resp.\ $\linktwo$).
\end{proposition}

Our proof of \fullref{prop:min_general} relies on convex surface theory applied 
to the Seifert surfaces $\seifone$ and $\seiftwo$. In preparation for this, we 
begin by establishing the following lemma, which extends the work of Boranda, 
Traynor, and Yan \cite{BorTraYan13:MinLeg} by placing their result in the 
context of Legendrian handle graphs.

\begin{lemma}[cf.\ {\cite[Lemma~3.3]{BorTraYan13:MinLeg}}]
  \label{lem:doub_stab}
  Let $\link$ be an oriented Legendrian link in a tight contact manifold $(Y, 
  \xi)$, and $\pstab \comp \nstab(\link)$ the result of successive negative and 
  positive stabilization on a component of $\link$. Then there is a handle 
  graph $\graph$ on $\pstab \comp \nstab (\link)$ such that $\Surg(\graph, 
  \pstab \comp \nstab (\link))$ is Legendrian isotopic to $\link$.
\end{lemma}

\begin{proof}
  The proof is essentially contained in \fullref{fig:doub_stab}, which 
  explicitly identifies a local model for the desired handle graph. \qedhere

  \begin{figure}[htbp]
    \labellist
    \small\hair 2pt
    \pinlabel {$\leg_-$} [ ] at 13 27
    \pinlabel {$\leg_+$} [ ] at 442 133
    \endlabellist
    \includegraphics[width=5in]{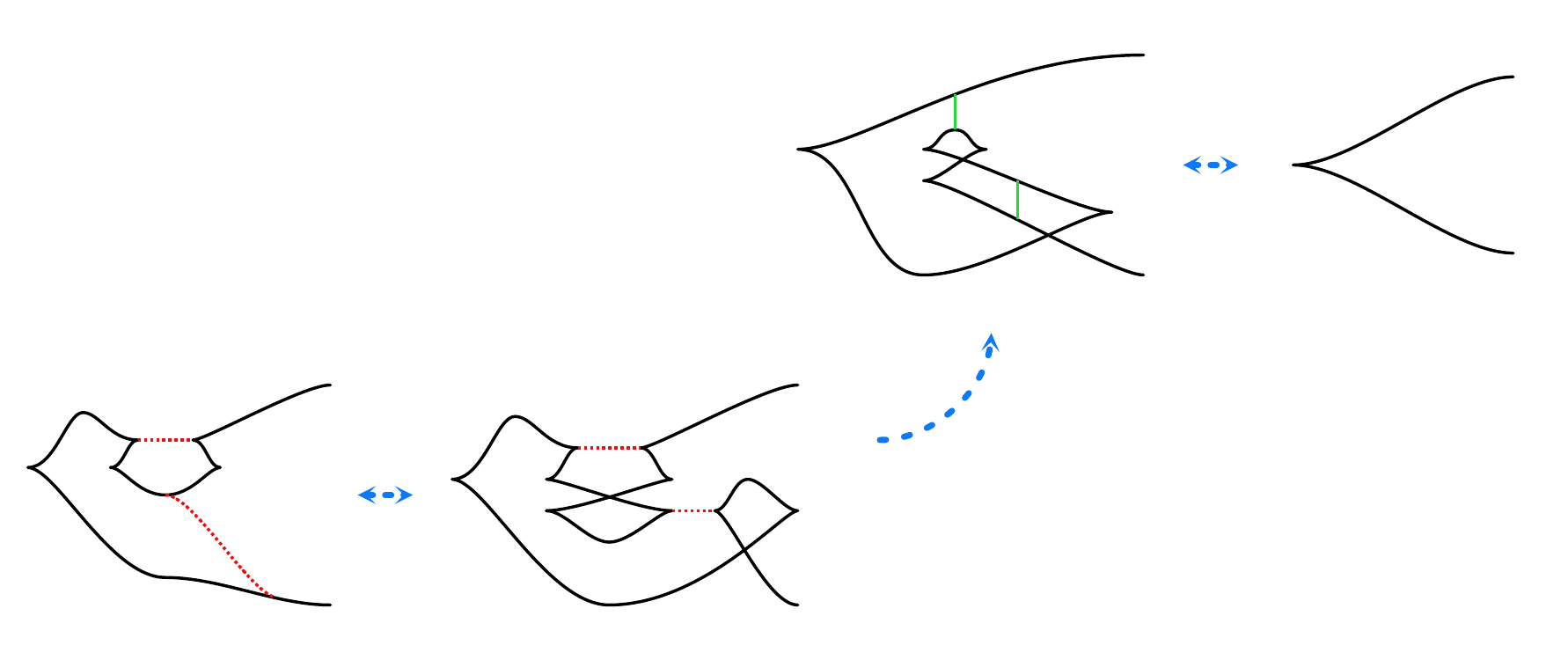}
    \caption{A handle graph giving rise to a Lagrangian cobordism from $S_+ 
      \circ S_- (\link)$ to $\link$.}
    \label{fig:doub_stab}
  \end{figure}
\end{proof}

\begin{lemma}
  \label{lem:unknot_cob}
  Let $\link$ be oriented, null-homologous Legendrian link in a tight contact 
  manifold $(Y, \xi)$. Then there exists an oriented Legendrian unknot $\unknot
  \subset (Y, \xi)$ and a handle graph $\graph$ on $\unknot$ such that 
  $\Surg(\graph, \unknot)$ is Legendrian isotopic to $\link$.
\end{lemma}

\begin{proof}
  Suppose that $\seif$ is a Seifert surface for $\link$. Applying 
  \fullref{lem:doub_stab} to successively double-stabilize each component of 
  $\link$ if necessary, we obtain a handle graph $(\graph_1, \link_1)$ and a 
  Seifert surface $\seif_1$ for $\link_1$ isotopic to $\seif$, such that the 
  twisting of $\xi$ relative to $\seif_1$ along each component of $\bdy \seif_1 
  = \link_1$ is negative, and $\Surg (\graph_1, \link_1)$ is Legendrian 
  isotopic to $\link$.  Below, we will denote this first condition by the 
  shorthand notation $\tw (\xi, \seif_1) < 0$, and similarly for other 
  surfaces.

  By work of Kanda \cite{Kan98:NonExactness}, since $\tw (\xi, \seif_1) < 0$, 
  there is an isotopy of $\seif_1$ relative to $\bdy \seif_1 = \link_1$ such 
  that the resulting surface $\seif_2$ is convex.  (While we will not use this, 
  we may assume that this isotopy is a $C^0$ perturbation near the boundary, 
  followed by a $C^\infty$ perturbation of the interior.) Further, by possibly
  Legendrian-isotoping the handle arcs of $\graph_1$, we obtain a handle graph
  $(\graph_2, \link_2 = \link_1)$ whose handle arcs $\graph_2 \setminus 
  \link_2$ intersect $\seif_2$ transversely in a finite number of points.

  To aid the discussion to follow, we picture the convex Seifert surface 
  $\seif_2$ in disk-band form; see \fullref{fig:ss_band_form}. Below, we shall 
  distinguish the bands from the disks, by fixing the disk-band decomposition.  
  (Note that \fullref{fig:ss_band_form} is an abstract diagram of $\seif_2$; as 
  $\seif_2$ is embedded in $Y$, the bands may be ``linked''.)  Since $\tw (\xi, 
  \seif_2) < 0$, the dividing set must intersect each component of $\link$.  

  \begin{figure}[htbp]
    \includegraphics[scale=1.0]{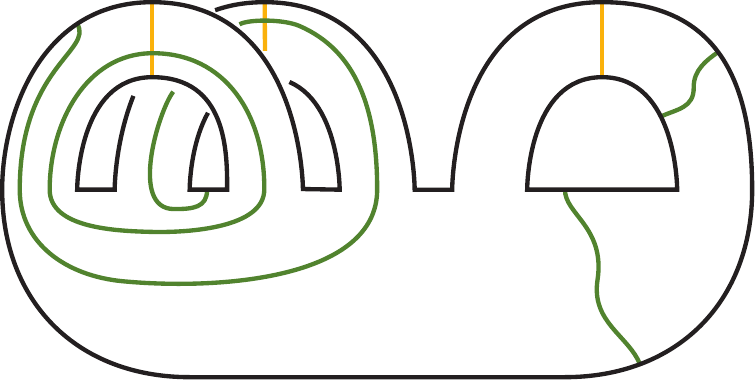}
    \caption{The convex Seifert surface $\seif_2$, dividing set 
      $\Gamma_\Sigma$, and arc basis, viewed in disk-band form.}
    \label{fig:ss_band_form}
  \end{figure}

  To obtain the desired handle graph $(\graph, \unknot)$, our strategy is to 
  cut the bands of $\seif_2$. Precisely, let $\set{a_1, \dotsc, a_g}$ be an arc 
  basis for $\seif_2$ consisting of a collection of properly embedded arcs in 
  $\seif_2$, such that the intersection of each $a_i$ with $\graph_2 \setminus 
  \link_2$ is empty.  \fullref{fig:band} depicts a band of $\seif_2$ and a 
  corresponding basis arc $a_i$.

  \begin{figure}[htbp]
    \includegraphics[scale=1.0]{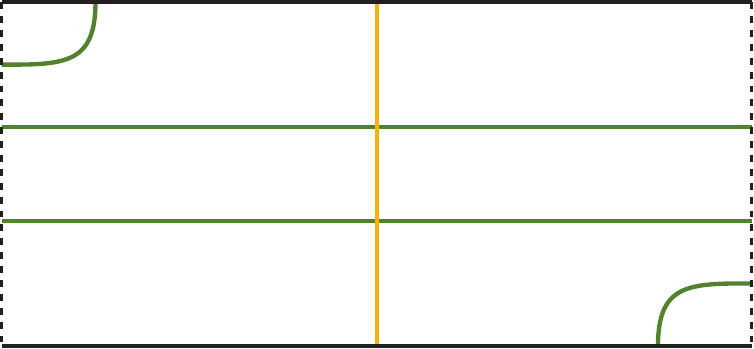}
    \caption{A band of the convex Seifert surface $\seif_2$ and a corresponding 
      basis arc $a_i$.}
    \label{fig:band}
  \end{figure}

  Now construct a (not necessarily Legendrian) link $\link_{2.5}$ as follows: 
  Take a parallel push-off of $\link_2$ in $\seif_2$ and, for each basis arc 
  $a_i$ that intersects the dividing set $\divset_{\seif_2}$, perform a finger 
  move across each of the dividing curves involved and back, as shown in 
  \fullref{fig:band_stabilize}. Since $a_i \intersect (\graph_2 \setminus 
  \link_2) = \emptyset$ for each $i$, we may assume that the finger moves avoid 
  all intersection points between $\graph_2 \setminus \link_2$ and $\seif_2$.

  \begin{figure}[htbp]
    \includegraphics[scale=1.0]{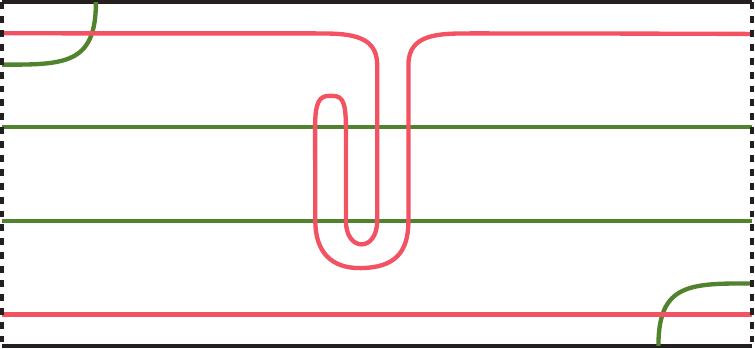}
    \caption{A double-stabilization $\link_3$ of $\link_2$ whose new Seifert 
      surface $\seif_3$ contains an arc basis disjoint from its dividing set.}
    \label{fig:band_stabilize}
  \end{figure}

  Since $\tw (\xi, \seif_2) < 0$, each component of $\seif_2 \setminus 
  \link_{2.5}$ intersects $\divset_{\seif_2}$, and we can apply the Legendrian 
  Realization Principle (LeRP) to $\link_{2.5} \subset \seif_2$ to obtain an 
  isotopy of $\seif_2$ to a convex surface $\seif_{2.5}$ with $\bdy \seif_{2.5} 
  = \link_2$, such that the image $\link_3$ of $\link_{2.5}$ under the isotopy 
  is Legendrian, and $\divset_{\seif_{2.5}}$ is the image of 
  $\divset_{\seif_2}$; see \cite{Kan98:NonExactness} and 
  \cite[Section~3]{Hon00:ClassificationI}. 

  Since the finger moves giving rise to $\link_{2.5}$---and hence 
  $\link_3$---each involved isotoping across elements of the dividing set an 
  even number of times, we have that the Legendrian link $\link_3$ is 
  necessarily an iterated double-stabilization of $\link_2$. In fact, $\link_3$ 
  is Legendrian isotopic to $\link_2$ outside of a tubular neighborhood $U_a$ 
  of the $a_i$'s.  We now construct a handle graph on $\link_3$ as follows: 
  First, extend the Legendrian isotopy between $\link_2 \setminus U_a$ and 
  $\link_3 \setminus U_a$ to a local contact isotopy, and apply the local 
  contact isotopy to the handle arcs in $\graph_2 \setminus \link_2$, obtaining 
  handle arcs that are attached to $\link_3$. Second, by 
  \fullref{lem:doub_stab}, we may add a collection $\handles_3$ of handle arcs 
  to this collection to obtain a handle graph $(\graph_3, \link_3)$ such that 
  $\Surg (\graph_3, \link_3, \handles_3)$ is Legendrian isotopic to $(\graph_2, 
  \link_2)$. In particular, this means that $\Surg (\graph_3, \link_3)$ is 
  Legendrian isotopic to $\link$. As before, by possibly Legendrian-isotoping 
  the handle arcs of $\graph_3$, we may assume that the handle arcs in of 
  $(\graph_3, \link_3)$ intersect $\seif_{2.5}$ transversely in a finite number 
  of points. Note that, by \fullref{fig:doub_stab} in the proof of 
  \fullref{lem:doub_stab}, the handle arcs of $\handles_3$ can be taken to be 
  contained in an arbitrarily small tubular neighborhood of the $a_i$'s, 
  implying that the complication in \fullref{rmk:handle-graph-limitation} does 
  not arise, since the handle arcs in $\handles_3$ are contained in a 
  neighborhood disjoint from $\graph_{2.5} \setminus \link_3$.)

  Let $\seif_3$ be the closure of the component of $\seif_{2.5} \setminus 
  \link_3$ that does not intersect $\bdy \seif_{2.5}$. Then we have obtained a 
  Legendrian link $\link_3$ bounding a convex Seifert surface $\seif_3$ that 
  contains an arc basis $\set{a_1', \dotsc, a_g'}$ that does not intersect the 
  dividing set $\divset_{\seif_3}$, and a handle graph $\graph_3$ on $\link_3$ 
  such that $\Surg (\graph_3, \link_3)$ is Legendrian isotopic to $\link$.

  We are now ready to construct the unknot $\unknot$ and desired Legendrian 
  handle graph $\graph$ on $\unknot$. We begin as above by taking a parallel 
  push-off of $\link_3$ in $\seif_3$ and topologically surgering it along the 
  basis arcs $a_i'$, to obtain a (not necessarily Legendrian) topological 
  unknot $\link_{3.5}$, together with a collection $\set{b_1, \dotsc, b_g}$ of 
  dual surgery arcs, as depicted in \fullref{fig:band_graph}. Again, we may 
  assume that $b_i \intersect (\graph_3 \setminus \link_3) = \emptyset$ for 
  each $i$.

  \begin{figure}[htbp]
    \includegraphics[scale=1.0]{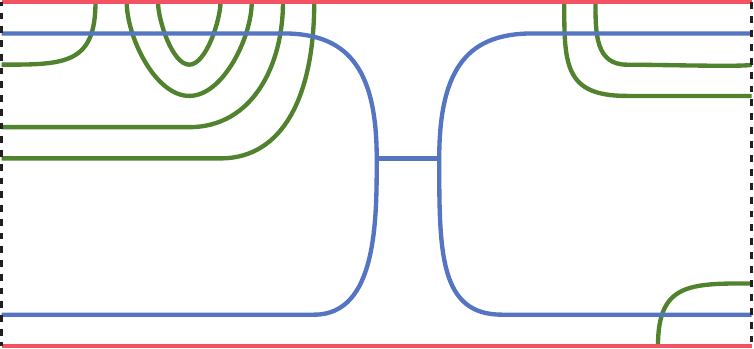}
    \caption{A Legendrian unknot $\unknot = \link_4$ contained in $\seif_3$ 
      together with dual surgery arcs $\set{b_1', \dotsc, b_g'}$.}
    \label{fig:band_graph}
  \end{figure}

  As above, since $\tw (\xi, \seif_3) < 0$ (as $\bdy \seif_3 = \link_3$ is a 
  double-stabilization of $\link_2 = \bdy \seif_2$, which has the same 
  property), each component of $\seif_3 \setminus (\link_{3.5} \union 
  \bigunion_i b_i)$ intersects $\divset_{\seif_3}$.  We may therefore again 
  apply the LeRP to $\link_{3.5} \union \bigunion_i b_i \subset \seif_3$ to 
  obtain an isotopy of $\seif_3$ to a convex surface $\seif_4$ with $\bdy 
  \seif_4 = \link_3$, such that the image $\link_4$ of $\link_{3.5}$ and the 
  image $b_i'$ of each $b_i$ under the isotopy is Legendrian, and 
  $\divset_{\seif_4}$ is the image of $\divset_{\seif_3}$.

  Again, $\link_4$ is Legendrian isotopic to $\link_3$ outside of a tubular 
  neighborhood $U_{a'}$ of the $(a_i')$'s; see 
  \cite[Section~3]{Hon00:ClassificationI}. We now construct a handle graph on 
  $\link_4$ as follows: First, extend the Legendrian isotopy between $\link_3 
  \setminus U_{a'}$ and $\link_4 \setminus U_{a'}$ to a local contact isotopy, 
  and apply the local contact isotopy to the handle arcs in $\graph_3 \setminus 
  \link_3$, obtaining handle arcs that are attached to $\link_4$. Second, we 
  add to this collection the collection $\handles_4 = \set{b_i'}$, to obtain a 
  handle graph $(\graph_4, \link_4)$ such that $\Surg (\graph_4, \link_4, 
  \handles_4)$ is Legendrian isotopic to $(\graph_3, \link_3)$. In particular, 
  this means that $\Surg (\graph_4, \link_4)$ is Legendrian isotopic to 
  $\link$.

  Let $\unknot = \link_4$ and $\graph = \graph_4$, and our proof is complete.
\end{proof}

We are now ready to prove the main result of this section. 

\begin{proof}[Proof of {\fullref{prop:min_general}}]
  According to \fullref{lem:unknot_cob}, there are oriented Legendrian unknots 
  $\unknotone$ and $\unknottwo$ and handle graphs $\graphone$ and $\graphtwo$, 
  such that $\Surg (\graphone, \unknotone)$ (resp.\ $\Surg (\graphtwo, 
  \unknottwo)$) is Legendrian isotopic to $\linkone$ (resp.\ $\linktwo$).

  Since $\rot_{[\seifone]} (\linkone) = \rot_{[\seiftwo]} (\linktwo)$, it 
  follows that $\rot (\unknotone) = \rot (\unknottwo)$. This also implies that 
  $\tb (\unknotone)$ and $\tb (\unknottwo)$ differ by a multiple of $2$. 
  Without loss of generality, assume that $\tb (\unknotone) \geq \tb 
  (\unknottwo)$; then by successively applying \fullref{lem:doub_stab} to 
  $\unknotone$ if necessary, we obtain a handle graph $(\graphoneb, 
  \unknotoneb)$ such that $\tb (\unknotoneb) = \tb (\unknottwo)$ and $\Surg 
  (\graphoneb, \unknotoneb)$ is Legendrian isotopic to $\graphone$.
  Again, by \fullref{fig:doub_stab} in the proof of \fullref{lem:doub_stab}, 
  the handle arcs of $\graphoneb$ can be taken to be contained in an 
  arbitrarily small neighborhood of a point; thus, we may combine these handle 
  arcs with those of $\graphone$, as in the proof of \fullref{lem:unknot_cob} 
  to obtain a handle graph $(\graphonet, \unknotoneb)$ such that $\Surg 
  (\graphonet, \unknotoneb)$ is Legendrian isotopic to $\linkone$.

  Now $\unknotoneb$ and $\unknottwo$ are unknots in the tight contact 
  $3$-manifold $(Y, \xi)$ with the same Thurston--Bennequin and rotation 
  numbers.  By the classification of Legendrian unknots in tight contact 
  manifolds by Eliashberg and Fraser \cite{EliFra98:UnknotTransSimple}, we have 
  that there is a contact isotopy $\phi_t$ of $(Y, \xi)$ taking $\unknotoneb$ 
  to $\unknottwo$.

  We now apply the isotopy $\phi_t$ to the Legendrian handle graph $\graphonet$ 
  and perturb the result so that the attached handles are disjoint from those 
  of $\graphtwo$. We then obtain a pair of Legendrian handle graphs for the 
  unknot $\unknottwo$, surgery along which yields Legendrian links isotopic to 
  $\linkone$ and $\linktwo$ respectively.
\end{proof}

%% file: sec_min_diag.tex
\section{Lower bounds via diagrams}
\label{sec:min_diag}

In this section, we re-prove \fullref{lem:unknot_cob} for Legendrian links in 
the standard contact $\R^3$ using diagrammatic techniques rather than convex 
surface theory.  This proof refines that of \cite{BorTraYan13:MinLeg} to 
produce a handle graph as well as a Lagrangian cobordism from an unknot.  We 
begin with a sequence of lemmas that reduce the number of crossings of the 
front diagram of the Legendrian link in a Legendrian handle graph at the 
expense of increasing the number of handles.  But first, we state a technical 
general position result.

\begin{lemma}
  \label{lem:graph-general-position}
  For any Legendrian handle graph $(\graph, \leg)$, there exists a $C^0$-close, 
  isotopic Legendrian handle graph $(\graph', \leg')$ such that all singular 
  points of the front diagram of $\graph$ have distinct $x$ coordinates.
\end{lemma}

\begin{proof}
  While this lemma simply expresses general position for the graph $\graph$, we 
  note in \fullref{fig:triple-pt-cusp}~(a) that moving a triple point off of a 
  cusp of $\leg$ is tantamount to using a Reidemeister VI move. \qedhere

  \begin{figure}[htbp]
    \labellist
    \tiny\hair 2pt
    \pinlabel {VI} [ ] at 86 88
    \pinlabel {VI} [ ] at 314 131
    \pinlabel {IV,II} [l] at 371 80
    \pinlabel {I} [ ] at 314 43
    \small
    \pinlabel {(a)} [ ] at 84 0
    \pinlabel {(b)} [ ] at 314 0
    \endlabellist
    \includegraphics[width=5in]{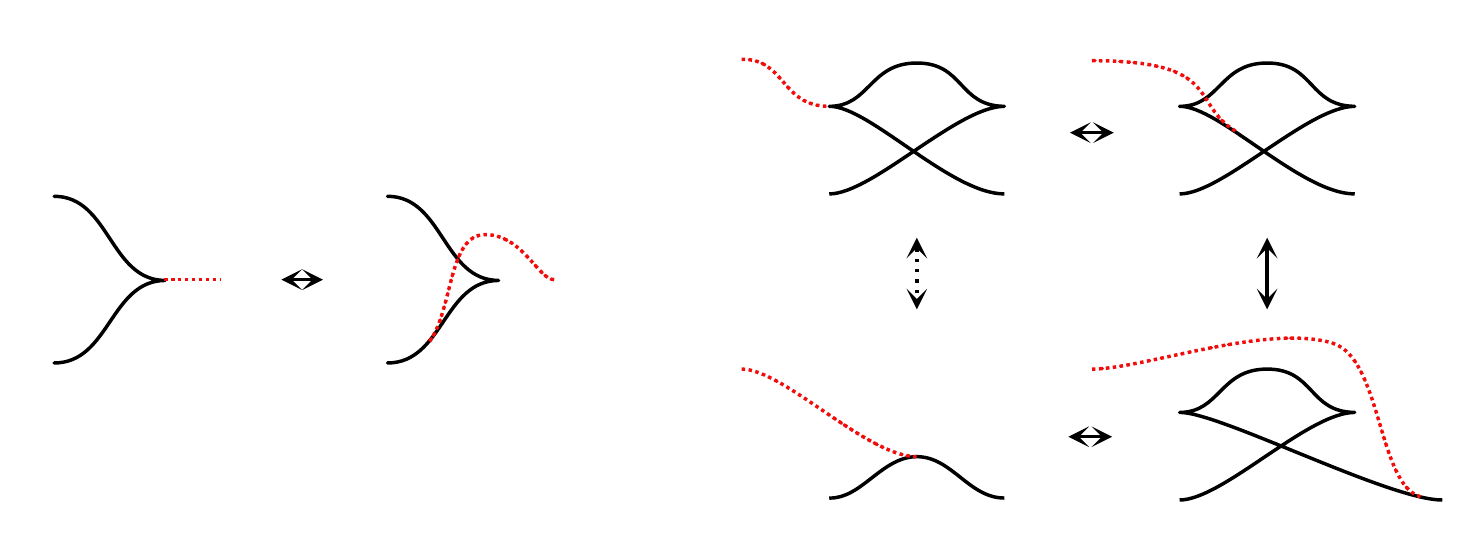}
    \caption{(a) Moving a triple point off of a cusp using a Reidemeister VI 
      move from \cite{ODoPav12:LegGraphs} and (b) Clearing a cusp of $\leg$ for 
      a Reidemeister I move.}
    \label{fig:triple-pt-cusp}
  \end{figure}
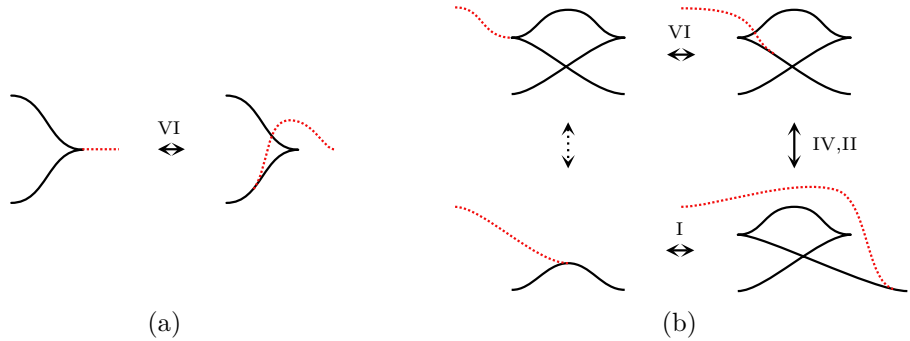
\end{proof}

First, we remove negative crossings.

\begin{lemma}
  \label{lem:negative-crossings}
  Given a Legendrian link $\leg_+$, whose front diagram has a negative 
  crossing, and a Legendrian handle graph $\graph_+$ on $\leg_+$, there exists 
  a Legendrian handle graph $(\graph_-,\leg_-)$ and a subset $\handles_0$ of 
  handles of $\graph_-$, such that $\Surg(\graph_-, \leg_-, \handles_0)$ is 
  Legendrian isotopic to $(\graph_+, \leg_+)$, and the front diagram of 
  $\leg_-$ has one fewer negative crossing than that of $\leg_+$.
\end{lemma}

\begin{proof}
  After applying \fullref{lem:graph-general-position} to isolate negative 
  crossings, the proof of \fullref{lem:negative-crossings} is contained in 
  \fullref{fig:negative-crossings}. \qedhere

  \begin{figure}[htbp]
    \labellist
    \small\hair 2pt
    \pinlabel {$(\graph_-, \leg_-)$} [ ] at 102 17
    \pinlabel {$(\graph_+, \leg_+)$} [l] at 420 285
    \endlabellist
    \includegraphics[width=5in]{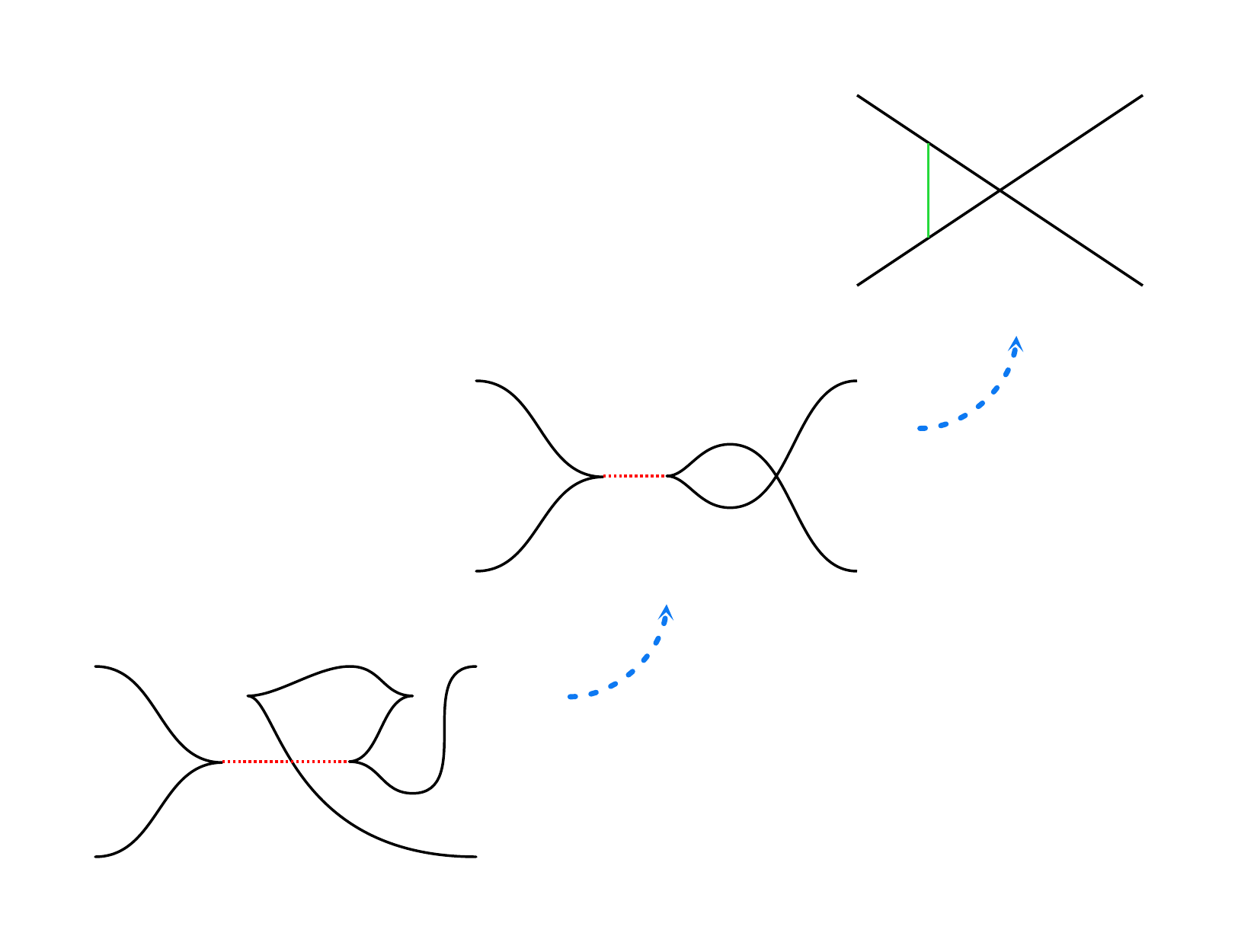}
    \caption{The Legendrian handle graph $(\graph_-,\leg_-)$ has one fewer 
      negative crossing than $(\graph_+,\leg_+)$. Red curves represent surgery 
      disks, i.e.\ cores of handles, while green curves represent co-cores.}
    \label{fig:negative-crossings}
  \end{figure}
\end{proof}

Next, we remove positive crossings.

\begin{lemma}
  \label{lem:positive-crossings}
  Given a Legendrian link $\leg_+$, the leftmost crossing of whose front 
  diagram is positive, and a Legendrian handle graph $\graph_+$ on $\leg_+$, 
  there exists a Legendrian handle graph $(\graph_-,\leg_-)$ and a subset 
  $\handles_0$ of handles of $\graph_-$, such that $\Surg(\graph_-, \leg_-, 
  \handles_0)$ is Legendrian isotopic to $(\graph_+, \leg_+)$ and the front 
  diagram of $\leg_-$ has one fewer positive crossing than that of $\leg_+$.
\end{lemma}

\begin{proof}
  Apply \fullref{lem:graph-general-position} to isolate crossings and cusps of 
  $\leg_+$ from handles of $\graph_+$. Consider the leftmost crossing $X_0$ of 
  $\leg_+$.  Without loss of generality, we may assume that $\leg_+$ is 
  oriented from right to left on both strands of $X_0$. The upper-left strand 
  incident to $X_0$ must thus next return to $X_0$; the same is true for the 
  bottom-left strand.  Since there are no crossings of $\leg_+$ to the left of 
  $X_0$, either the upper left strand  must next cross the $x$-coordinate of 
  $X_0$ above $X_0$ or the lower left strand must next cross the $x$-coordinate 
  of $X_0$ below $X_0$.  Without loss of generality, assume that this holds for 
  the upper left strand as in the upper-right portion of 
  \fullref{fig:positive-crossing}. Let $-\strand \subset \leg_+$ be the compact 
  $1$-manifold that starts at $X_0$ and traverses along the upper-left strand 
  of $X_0$ until returning to the same $x$-coordinate, and let $\strand$ be 
  $-\strand$ with the orientation reversed, so that $X_0$ is at the end of 
  $\strand$.

  As in the second diagram down in \fullref{fig:positive-crossing}, create a 
  finger of $\leg_+$ parallel to $\strand$ using a Reidemeister I move at the 
  initial point of $\strand$ and next to every cusp of $\strand$ along with 
  Reidemeister II moves to pass the lead cusp of the finger through handles of 
  $\graph_+$ that are incident to $\strand$.  Move the end of the finger just 
  to the right of $X_0$.  Place a co-core of a handle inside the finger just 
  below each crossing created by a Reidemeister I move.  Place two additional 
  co-cores from the finger to the original link on either side of the crossing 
  $X_0$.

  \begin{figure}[htbp]
    \labellist
    \small\hair 2pt
    \pinlabel {$(\graph_+, \leg_+)$} [ ] at 255 507
    \pinlabel {$(\graph_-, \leg_-)$} [ ] at 78 45
    \endlabellist
    \includegraphics[width=5in]{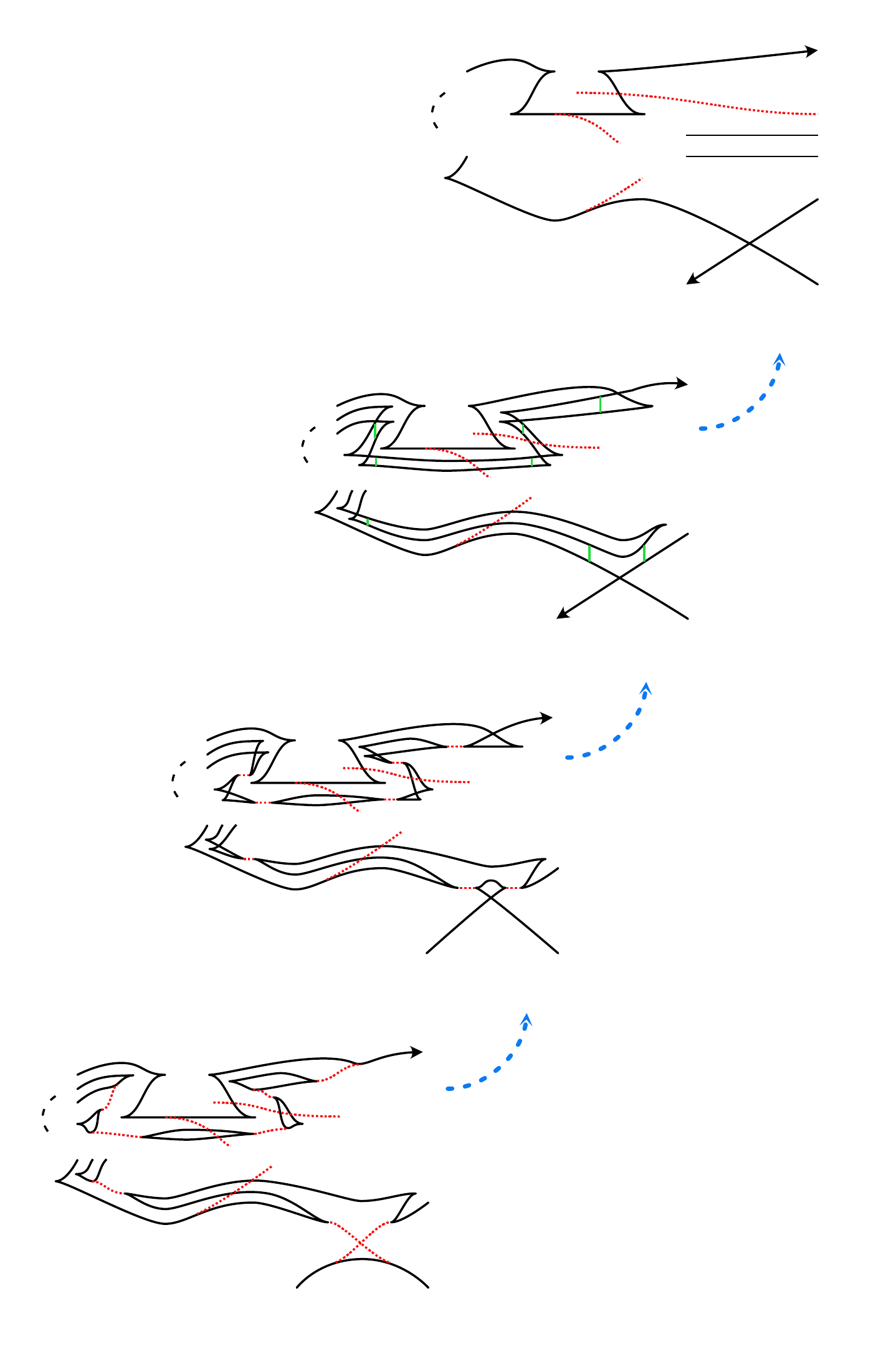}
    \caption{The Legendrian handle graph $(\graph_-,\leg_-)$ has one fewer 
      positive crossing than $(\graph_+,\leg_+)$. Red curves represent cores of 
      handles, while green curves represent co-cores.}
    \label{fig:positive-crossing}
  \end{figure}

  Finally, replace the co-cores by surgery disks to create a new Legendrian 
  handle graph as in the third row of \fullref{fig:positive-crossing}.  Isotope 
  the new Legendrian handle graph as at the bottom of 
  \fullref{fig:positive-crossing}, using a combination of the move in 
  \fullref{fig:triple-pt-cusp}~(b) to move the handles away and Reidemeister I 
  moves to remove the crossings.  The result is a Legendrian handle graph that 
  has many more surgery disks, but whose underlying Legendrian link has one 
  fewer crossing than before.
\end{proof}

The procedure above may produce a disconnected Legendrian link.  We next see 
how to join these components.

\begin{lemma}
  \label{lem:combine-components}
  Let $(\graph_+, \leg_+)$ be a Legendrian handle graph, where $\leg_+$ has $n 
  \geq 2$ components, which are mutually disjoint in the front diagram.  
  Suppose that there exists a path $\smpath$ in the front diagram of $\graph_+$ 
  that starts on component $\leg_+' \subset \leg_+$, ends on $\leg_+'' \subset 
  \leg_+$, and does not intersect $\leg_+$ otherwise.  Then there exists a 
  Legendrian handle graph $(\graph_-,\leg_-)$ and a subset $\handles_0$ of 
  handles of $\graph_-$, such that $\Surg(\graph_-, \leg_-, \handles_0)$ is 
  Legendrian isotopic to $(\graph_+,\leg_+)$, one component of $\leg_-$ is 
  topologically the connected sum of $\leg_+'$ and $\leg_+''$, the other 
  components of $\leg_-$ match the remaining components of $\leg_+$, and none 
  of the components of $\leg_-$ intersect in the front diagram.
\end{lemma}

\begin{proof}
  We may assume that $\smpath$ intersects $\leg_+'$ and $\leg_+''$ away from 
  triple points, crossings, and cusps.	Create a finger of $\leg_+'$ that 
  follows $\smpath$, starting with a Reidemeister I move and using Reidemeister 
  II moves to cross handles of $\graph_+$ and additional Reidemeister I moves 
  when $\smpath$ has a vertical tangent; see the middle diagram of 
  \fullref{fig:combine-components}.  Stop the finger just before $\smpath$ 
  intersects $\leg_+''$, performing an additional Reidemeister I move if 
  necessary to ensure that the orientations of parallel strands of the finger 
  and $\leg_+''$ are opposite. Place a co-core of a handle between those two 
  parallel strands.  Finally, replace the co-core by a core of a handle to 
  create a new Legendrian handle graph $(\graph_-, \leg_-)$ as in the 
  bottom-left portion of \fullref{fig:combine-components}.

  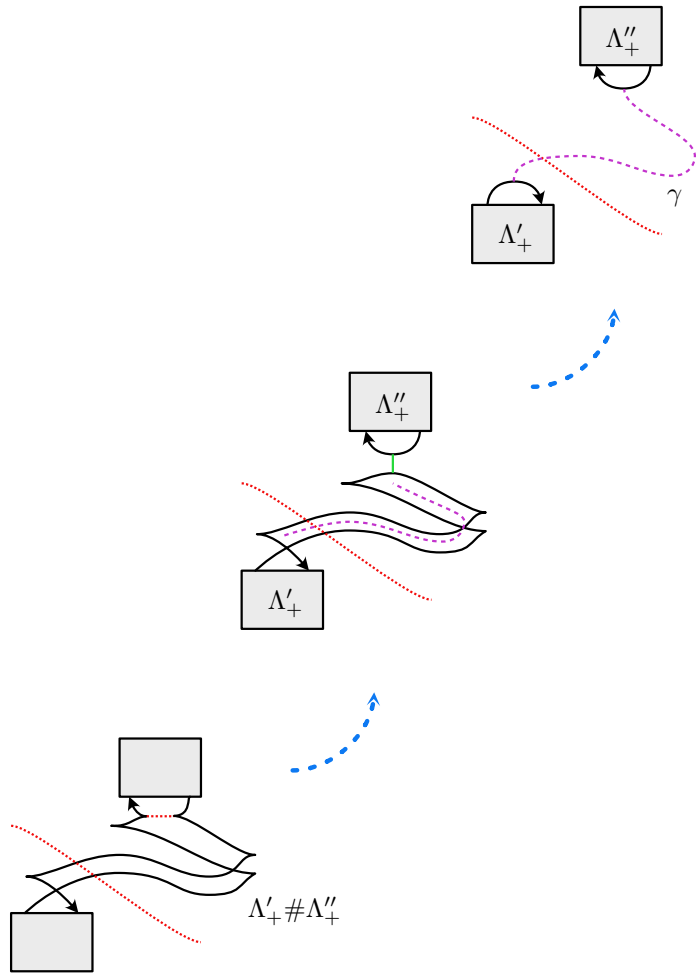
\begin{figure}[htbp]
    \labellist
    \small\hair 2pt
    \pinlabel {$\leg''_+$} [ ] at 292 435
    \pinlabel {$\leg'_+$} [ ] at 244 347
    \pinlabel {$\leg''_+$} [ ] at 189 271
    \pinlabel {$\leg'_+$} [ ] at 141 183
    \pinlabel {$\leg'_+ \# \leg''_+$} [ ] at 145 45
    \pinlabel {$\smpath$} [ ] at 315 365
    \endlabellist
    \includegraphics[width=4in]{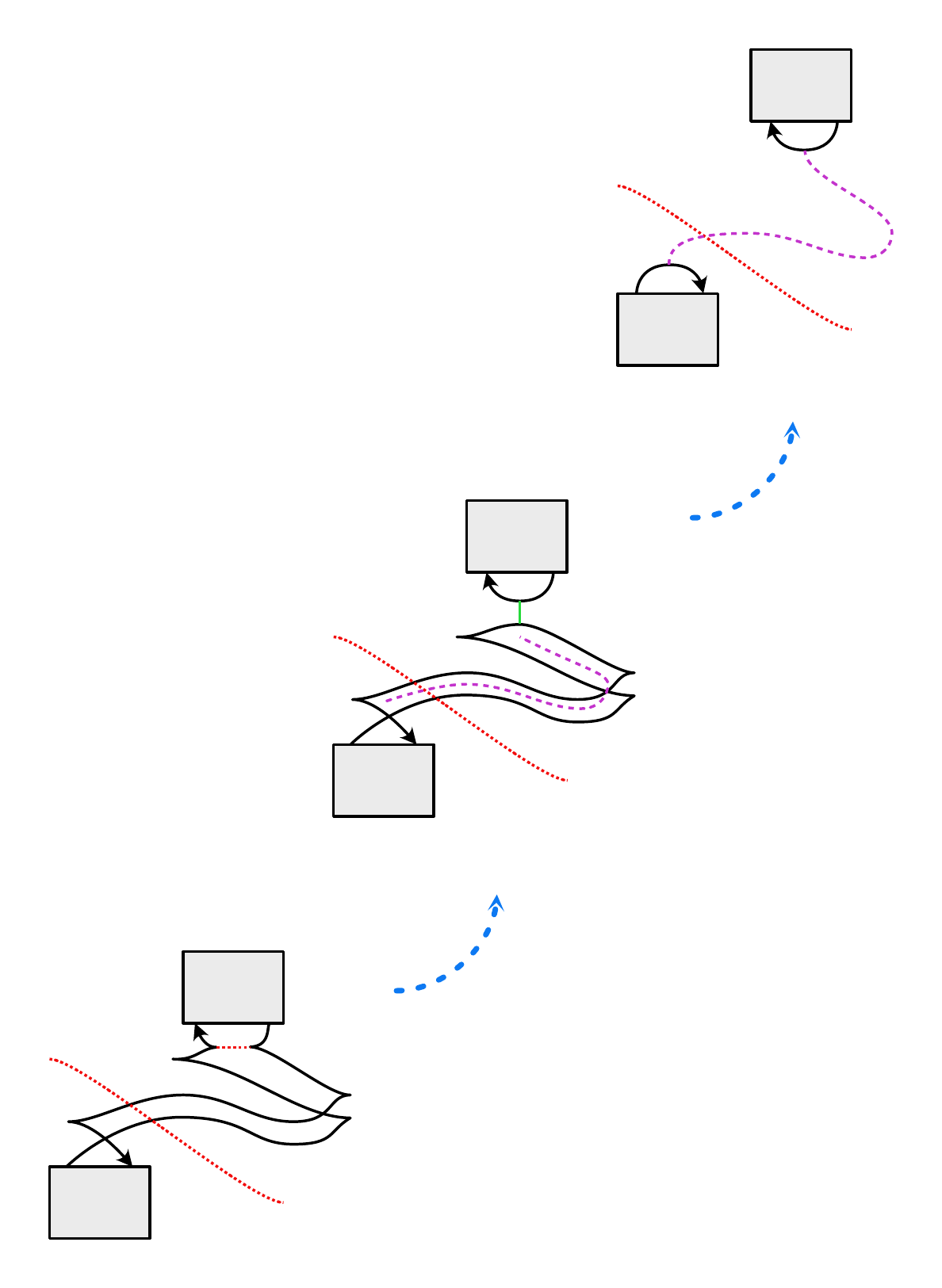}
    \caption{The Legendrian link $\leg_-$ has a component that is topologically 
      the connected sum of two components of $\leg_+$.}
    \label{fig:combine-components}
  \end{figure}

  That the new component of $\leg_-$ is the connect sum of $\leg_+'$ and 
  $\leg_+''$ comes from the facts that the diagrams of $\leg_+'$ and $\leg_+''$ 
  are disjoint and that $\smpath$ is disjoint from the diagram of $\leg_+$ on 
  its interior.	The final two conclusions of the lemma follow immediately from 
  the construction.
\end{proof}

With the tools above in place, we are ready to re-prove 
\fullref{lem:unknot_cob} using the diagrammatic techniques of this section.

\begin{proof}[Diagrammatic proof of {\fullref{lem:unknot_cob}} in $(\R^3, 
  \xistd)$]
  Given a Legendrian $\link$, use \fullref{lem:negative-crossings} repeatedly, 
  and then \fullref{lem:positive-crossings} repeatedly, to obtain a Legendrian 
  handle graph $(\graph_1, \leg_1)$ such that the front diagram of $\leg_1$ has 
  no crossings, and $\Surg (\graph_1, \leg_1)$ is Legendrian isotopic to 
  $\link$.  Use \fullref{lem:combine-components} to find a Legendrian handle 
  graph $(\graph_2, \leg_2)$ with a subset $\handles_0$ of handles, such that 
  $\leg_2$ is connected, and $\Surg (\graph_2, \leg_2, \handles_0)$ is 
  Legendrian isotopic to $\Surg (\graph_1, \leg_1)$, which implies that $\Surg 
  (\graph_2, \leg_2)$ is Legendrian isotopic to $\link$.  Finally, note that 
  $\leg_2$ is a smooth unknot since it is the connected sum of smooth unknots.
\end{proof}

%% file: sec_max.tex
\section{Upper bounds}
\label{sec:max}

With constructions of a common lower bound and corresponding handle graphs for 
$\legone$ and $\legtwo$ in hand, we are ready to find an upper bound.  The 
structure of the following proof parallels that of Lazarev 
\cite{Laz20:MaxContSymp} in higher dimensions.

\begin{proof}[Proof of {\fullref{prop:slices}}]
  Given oriented Legendrian links in $\legone$ and $\legtwo$ in $(Y, \xi)$, 
  \fullref{prop:min_general} implies that there exist an oriented Legendrian 
  link $\legm$ and Legendrian handle graphs $(\graphone, \legm)$ and 
  $(\graphtwo, \legm)$ such that $\Surg (\graphone, \legm)$ (resp.\ $\Surg 
  (\graphtwo, \legm)$) is Legendrian isotopic to $\legone$ (resp.\ $\legtwo$).

  We Legendrian isotope the handles $\handlestwo$ of $\graphtwo$ to be in 
  general position with respect to the handles $\handlesone$ of $\graphone$.  
  In particular, we may assume that the Legendrian handle graph $(\graphtwo, 
  \legm)$ has $\handlestwo$ is disjoint from $\handlesone$, with $\Surg 
  (\graphtwo, \legm)$ still Legendrian isotopic to $\legtwo$.

  Define the Legendrian graph $\graph_+ = \graphone \union \graphtwo$; it is 
  clear that $(\graph_+, \legm)$ is a Legendrian handle graph. Note that $\Surg 
  (\graph_+, \legm, \handlesone)$ is Legendrian isotopic to a Legendrian handle 
  graph $(\graph_{+, 1}, \legone)$; similarly, $\Surg (\graph_+, \legm, 
  \handlestwo)$ is Legendrian isotopic to a Legendrian handle graph 
  $(\graph_{+, 2}, \legtwo)$.

  Observe that both $\Surg (\graph_{+, 1}, \legone)$ and $\Surg (\graph_{+, 2}, 
  \legtwo)$ are Legendrian isotopic to $\Surg (\graph_+, \legm)$, which we 
  denote by $\legp$. Let $\cobone \colon \legm \to \legp$ be the concatenation 
  of $\cob (\graphone, \legm)$ with $\cob (\graph_{+, 1}, \legone)$; similarly, 
  let $\cobtwo \colon \legm \to \legp$ be the concatenation of $\cob 
  (\graphtwo, \legm)$ with $\cob (\graph_{+, 2}, \legtwo)$. Then it is clear 
  that $\legone$ (resp.\ $\legtwo$) appears as a collared slice of $\cobone$ 
  (resp.\ $\cobtwo$).  At the same time, \fullref{prop:graph-surg-order} 
  implies that $\cobone$ and $\cobtwo$ are exact-Lagrangian isotopic, since 
  they are both obtained from the same Legendrian handle graph $(\graph_+, 
  \legm)$ by Legendrian ambient surgery, only in a different order---in other 
  words, they both belong to the isotopy class $\cob (\graph_+, \legm)$.
\end{proof}

\begin{example}
  \fullref{fig:trefoil-52-full-process} and 
  \fullref{fig:fig8-unknot-full-process} display the full process of creating 
  the upper bounds in \fullref{fig:main-example1} and 
  \fullref{fig:main-example2}, respectively.

  \begin{figure}[htbp]
    \includegraphics{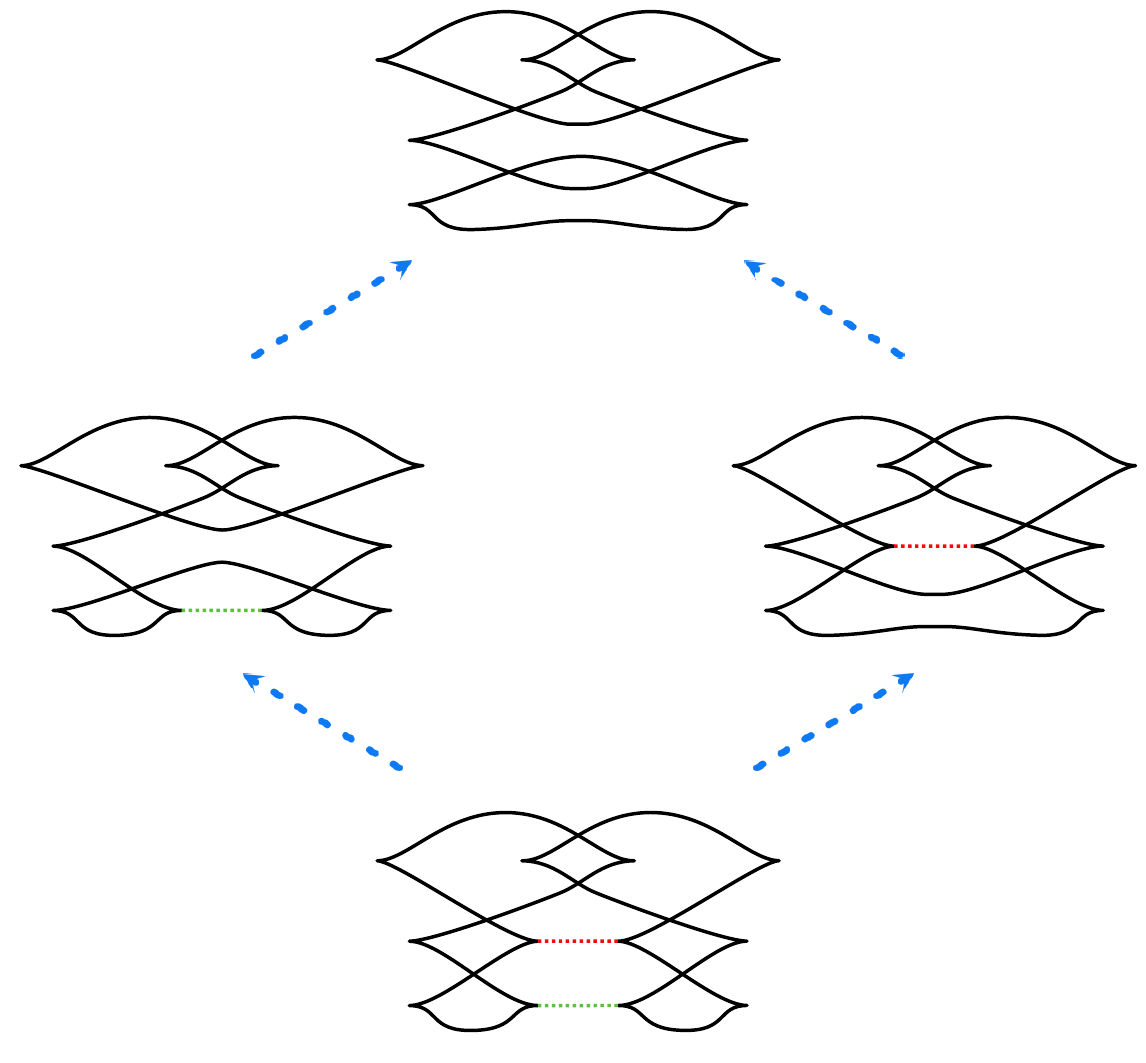}
    \caption{The handle graph at the bottom of the figure is used to create the 
      upper bound of the trefoil and an $m(5_2)$ knot that appeared in 
      \fullref{fig:main-example1}.}
    \label{fig:trefoil-52-full-process}
  \end{figure}

  \begin{figure}[htbp]
    \includegraphics{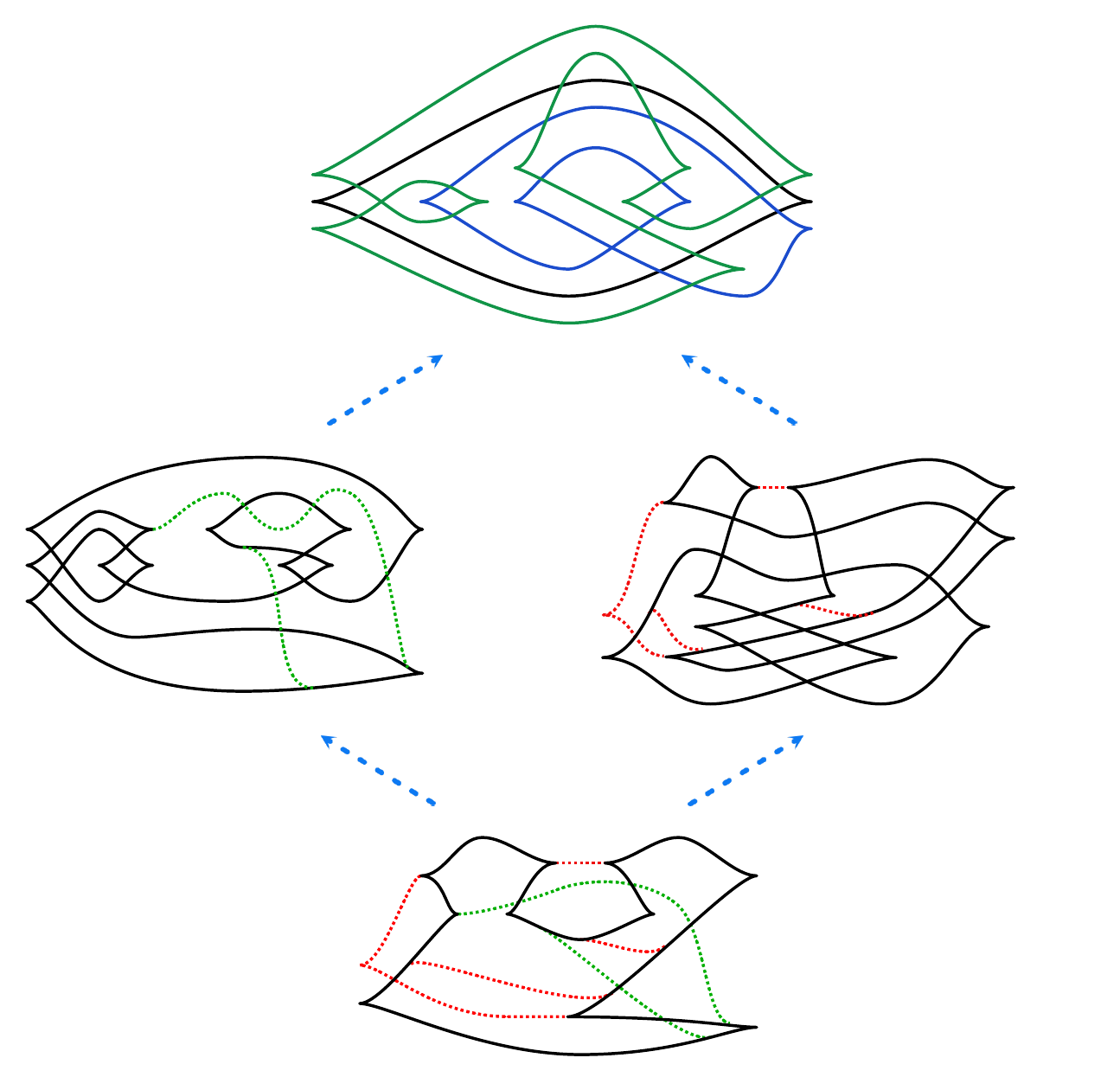}
    \caption{The handle graph at the bottom of the figure is used to create the 
      upper bound of the figure eight knot and the unknot that appeared in 
      \fullref{fig:main-example2}.  Note that the Legendrian knots in the 
      handle graphs in the middle level are isotopic to the unknot (left) and 
      the figure eight (right).}
    \label{fig:fig8-unknot-full-process}
  \end{figure}
\end{example}

%% file: sec_genus.tex
\section{The Lagrangian Cobordism Genus}
\label{sec:genus}

In this section, we use the construction of upper and lower bounds for a pair 
of Legendrian knots to define a new quantity, the relative Lagrangian genus, 
and a new relation, Lagrangian quasi-concordance.  We explore foundational 
properties and immediate examples, leaving deeper explorations, as embodied in 
the list of open questions at the end, for future work.  For ease of notation, 
we work with Legendrian links in the standard contact $\R^3$, though our 
definitions may easily be adapted to Legendrians in any tight contact 
$3$-manifold.

\subsection{Lagrangian Quasi-Cobordism}
\label{ssec:quasi-cob}

We begin with a definition that undergirds the two concepts referred to above.

\begin{definition}
  \label{defn:quasi-cob}
  A \emph{Lagrangian quasi-cobordism} between Legendrian links $\leg$ and 
  $\leg'$ consists of an ordered set of $n+1$ Legendrian links
  \[
    \legsint = (\leg = \leg_0, \leg_1, \dotsc, \leg_n = \leg'),
  \]
  another ordered set of $n$ nonempty Legendrian links
  \[
    \legsuplow = (\leg^*_1, \dotsc, \leg^*_n),
  \]
  such that $\leg^*_i$ is an upper or lower bound for the pair 
  $(\leg_{i-1},\leg_i)$, and the Lagrangian cobordisms
  \[
    \cobsint = (\cob_1^<, \cob_1^>, \cob_2^<, \cob_2^>, \dotsc, \cob_n^<, 
    \cob_n^>)
  \]
  that realize the upper or lower bound constructions.
\end{definition}

There are several quantities associated to a Lagrangian quasi-cobordism.  

\begin{definition}
  \label{defn:quasi-genus}
  Given a Lagrangian quasi-cobordism $(\legsint, \legsuplow, \cobsint)$, its 
  \emph{length} is one less than the number of elements in $\legsint$, while 
  its \emph{Euler characteristic} $\chi (\legsint, \legsuplow, \cobsint)$ is 
  the sum of the Euler characteristics of the Lagrangians in $\cobsint$ and its 
  \emph{genus} $g (\legsint, \legsuplow, \cobsint)$ defined, as usual, in terms 
  of the Euler characteristic. 

  Further, we define the \emph{relative Lagrangian genus} $g_L (\leg, \leg')$ 
  between the Legendrian links $\leg$ and $\leg'$ as the minimum genus of any 
  Lagrangian quasi-cobordism between them. Two Legendrian links $\leg$ and 
  $\leg'$ are \emph{Lagrangian quasi-concordant} if $g_L (\leg, \leg') = 0$.
\end{definition}

\begin{example}
  \label{ex:quasi-cob-not-cob}
  Let $\maxunknot$ be the maximal Legendrian unknot, and let $\leg$ be a 
  maximal Legendrian representative of $m (6_2)$.  Note that both $\maxunknot$ 
  and $\leg$ have Thurston--Bennequin number $-1$ and that the smooth $4$-genus 
  of $6_2$ is equal to $1$ \cite{LivMoo:KnotInfo}.  It follows from the 
  behavior of the Thurston--Bennequin invariant under Lagrangian cobordism that 
  there cannot be a Lagrangian cobordism joining $\maxunknot$ and $\leg$ in 
  either direction.  Nevertheless, there is a genus-$1$ Lagrangian 
  quasi-cobordism between the two; see \fullref{fig:quasi-cob-6-2}.

  \begin{figure}[htbp]
    \includegraphics[width=4.5in]{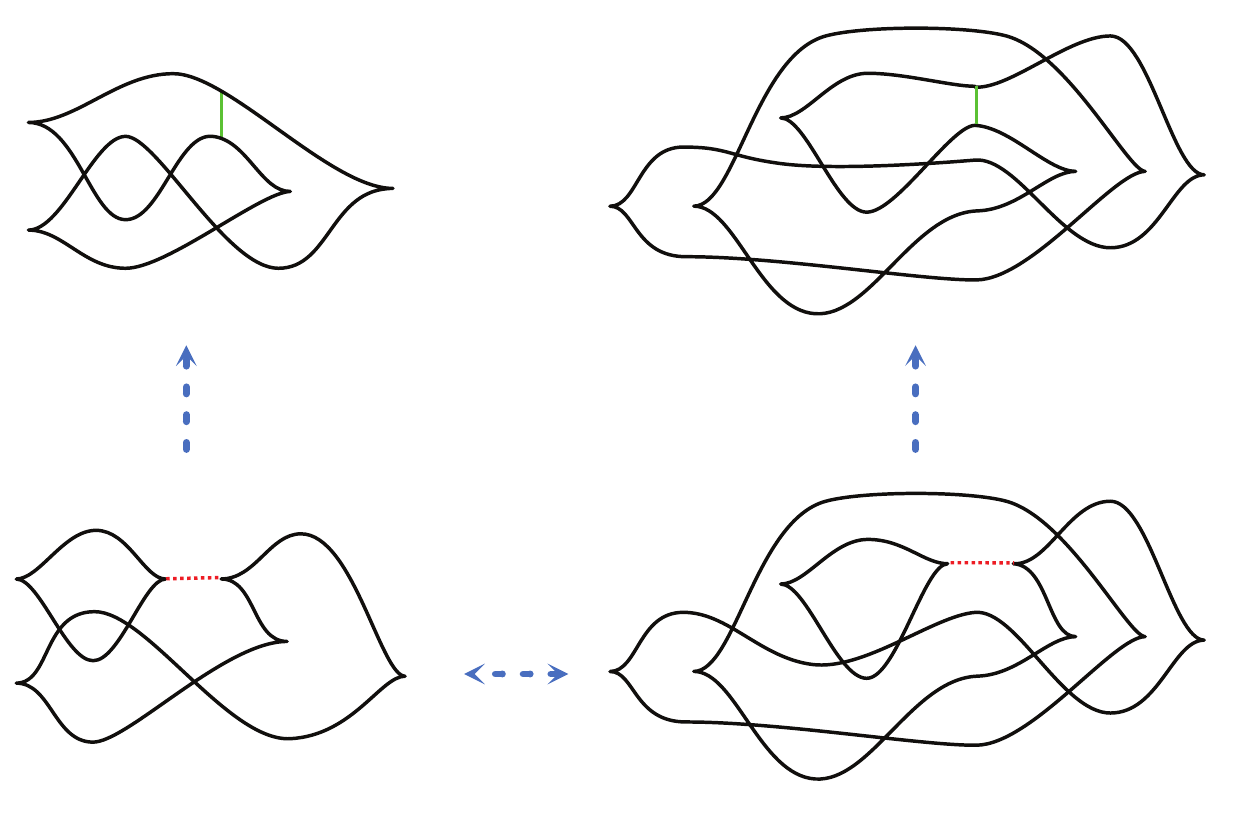}
    \caption{A genus $1$ Lagrangian quasicobordism between the maximal unknot 
      $\maxunknot$ and a maximal representative $\leg$ of the mirror of the 
      $6_2$ knot.  The quasi-cobordism was produced using the ideas in 
      \cite[Section~5]{BorTraYan13:MinLeg}, especially Figure~25 and 
      Figure~27.}
    \label{fig:quasi-cob-6-2}
  \end{figure}
\end{example}

Lagrangian quasi-cobordism induces an equivalence relation on the set of 
isotopy classes of Legendrian links.  As in the smooth case, this equivalence 
relation is uninteresting, as shown by the following immediate corollary of 
\fullref{thm:main} or \fullref{prop:min_general}:

\begin{corollary}
  \label{cor:same-r-quasi-cob}
  Any two Legendrian links with the same rotation number are Lagrangian 
  quasi-cobordant.  In fact, the quasi-cobordism may be chosen to have length 
  $1$.
\end{corollary}

The corollary shows that the relative Lagrangian genus is defined for any two 
Legendrian links of the same rotation number.

On the other hand, Lagrangian quasi-concordance also clearly induces an 
equivalence relation on the set of isotopy classes of Legendrian links.  The 
relative Lagrangian genus descends to Lagrangian quasi-concordance classes.  
Both the rotation number and the Thurston--Bennequin number 
\cite{Cha10:LagConc} are invariants of Lagrangian quasi-concordance, though 
non-classical invariants coming from Legendrian Contact Homology or Heegaard 
Floer theory will have a more complicated relationship with quasi-concordance.

\subsection{Relation to Smooth Genus}
\label{ssec:lag-genus-sm-genus}

To connect the relative Lagrangian genus to smooth constructions, note that we 
may define the smooth cobordism genus between two smooth knots $K_1$ and $K_2$ 
to be the minimum genus of all cobordisms between them; we denote this by $g_s 
(K_1, K_2)$.  Chantraine proved that Lagrangian \emph{fillings} minimize the 
smooth $4$-ball genus of a Legendrian knot \cite{Cha10:LagConc}, and so one 
might ask if this minimization property extends to $g_L$.  We begin with a 
simple lemma.

\begin{lemma}
  \label{lem:basic-genus-comparison}
  Given Legendrian knots $\leg$ and $\leg'$, we have $g_s (\leg, \leg') \leq 
  g_L (\leg, \leg')$.
\end{lemma}

\begin{proof}
  Let $(\legsint, \legsuplow, \cobsint)$ be a Lagrangian quasi-cobordism 
  between $\leg$ and $\leg'$. Assume for ease of notation that each $\leg^* \in 
  \legsuplow$ is an upper bound. Let $\overline{\cob}_i^>$ be the smooth 
  cobordism from $\leg^*_i$ to $\leg_i$ obtained from reversing $\cob_i^>$; 
  note that $\overline{\cob}_i^>$ is not, in general, a Lagrangian cobordism.  
  Since Euler characteristic is additive under gluing, the smooth cobordism 
  $\cob_1^< \circ \overline{\cob}_1^> \circ \cob_2^< \circ \cdots \circ 
  \overline{\cob}_n^>$ has genus $g(\legsint,\legsuplow,\cobsint)$, and hence 
  $g_s(\leg, \leg') \leq g_L(\leg, \leg')$.
\end{proof}

It is natural to ask under what conditions on $\leg_1$ and $\leg_2$---as 
Legendrian or as smooth knots---is the inequality in 
\fullref{lem:basic-genus-comparison} an equality? On one hand, we cannot expect 
to achieve equality in all cases.  

\begin{example}
  Let $\leg$ be any Legendrian knot, and let $\leg'$ be a double stabilization 
  of $\leg$ with the same rotation number as $\leg$.  Since $\leg$ and $\leg'$ 
  have the same underlying smooth knot type, we have $g_s (\leg, \leg') = 0$.  
  On the other hand, let $(\legsint, \legsuplow, \cobsint)$ be a Lagrangian 
  quasi-cobordism between $\leg$ and $\leg'$. Note that $\chi (\cob_i^<), \chi 
  (\cob_I^>) \leq 0$ for all $i$, since each of $\cob_i^<$ and $\cob_i^>$ has 
  at least two boundary components. Since $\tb(\leg) > \tb(\leg')$, some pair 
  $\leg_i$, $\leg_{i+1}$ in $\legsint$ must have different Thurston--Bennequin 
  numbers.  In particular, the bound $\leg^*_i$ must have a different 
  Thurston--Bennequin number than at least one of $\leg_i$ or $\leg_{i+1}$.  It 
  follows that $\chi(\cob_i^<) + \chi(\cob_i^>) < 0$, and hence that $\chi 
  (\legsint, \legsuplow, \cobsint) < 0$. Since $\leg$ and $\leg'$ are knots, 
  this implies that $g(\legsint,\legsuplow,\cobsint) > 0$.  In particular, we 
  have $g_L(\leg, \leg') > 0$ even though $g_s(\leg, \leg') = 0$.
\end{example}

On the other hand, there is a simple sufficient condition for equality in the 
lemma above.

\begin{lemma}
  \label{lem:genus-comparison}
  If the Legendrian knot $\leg$ has a Lagrangian filling, and there exists a 
  Lagrangian cobordism from $\leg$ to $\leg'$, then $g_s(\leg, \leg') = 
  g_L(\leg, \leg')$.
\end{lemma}

\begin{proof}
  We begin by setting notation.  Let $\cob_0$ be the Lagrangian filling of 
  $\leg$ and let $\cob_1^<$ be the Lagrangian cobordism from $\leg$ to $\leg'$.  
  Taking $\cob_1^>$ to be the trivial cylindrical Lagrangian cobordism from 
  $\leg'$ to itself, and taking $\leg_1^* = \leg'$, we see that 
  \begin{equation}
    \label{eq:min-genus1}
    g_L (\leg, \leg') \leq g (\cob_1^<).
  \end{equation}

  Let $\Sigma$ be the smooth cobordism from $\leg$ to $\leg'$ that minimizes 
  the smooth cobordism genus. We know that $\cob_0 \circ \cob_1^<$ is a 
  Lagrangian filling of $\leg'$, and hence that $g (\cob_0 \circ \cob_1^<) \leq 
  g (\cob_0 \circ \Sigma)$.  Since $\leg$ is a knot, the genus is additive 
  under composition of cobordisms, and we obtain
  \begin{equation}
    \label{eq:min-genus2}
    g (\cob_1^<) \leq g (\Sigma).
  \end{equation}

  Combining \eqref{eq:min-genus1} and \eqref{eq:min-genus2}, we obtain 
  \[g_L(\leg, \leg') \leq g(\Sigma) = g_s(\leg, \leg').\]
  The lemma now follows from \fullref{lem:basic-genus-comparison}.
\end{proof}

\subsection{Open Questions}
\label{ssec:open-q}

We end with a list of questions about Lagrangian quasi-cobordism and 
quasi-concordance beyond the motivating question above about the relationship 
between the relative Lagrangian genus and the relative smooth genus.

\begin{enumerate}
  \item Building off of \fullref{ex:quasi-cob-not-cob}, is there an example of 
    a pair $\leg$ and $\leg'$ that are Lagrangian quasi-concordant but not 
    Lagrangian concordant?
  \item Taking the previous question further, for two Lagrangian 
    quasi-concordant Legendrians $\leg$ and $\leg'$, what is the minimal length 
    of any Lagrangian quasi-concordance between them?  Are there examples for 
    which this minimal length is arbitrarily high?
  \item Even more generally, define $g_L(\leg, \leg', n)$ to be the minimal 
    genus of any Lagrangian quasi-cobordism between $\leg$ and $\leg'$ of 
    length at most $n$.  The sequence $(g_L(\leg, \leg', n))_{n=1}^\infty$ 
    decreases to and stabilizes at $g_L(\leg, \leg')$.  Are there examples for 
    which the number of steps it takes the sequence to stabilize is arbitrarily 
    long?
  \item Can $g_L(\leg, \leg') - g_s (\leg, \leg')$ be arbitrarily large when 
    $\leg$ and $\leg$ both have maximal Thurston--Bennequin invariant?
  \item Can the hypotheses of \fullref{lem:genus-comparison} be weakened to 
    $\leg$ having only an augmentation instead of a filling?
  \item Is there a version of this theory for Maslov $0$ Lagrangians, which 
    would better allow the use of Legendrian Contact Homology, especially the 
    tools in \cite{Pan17:LagCobAug}?
\end{enumerate}

%% file: maxleg.bbl
\providecommand{\bysame}{\leavevmode\hbox to3em{\hrulefill}\thinspace}
\providecommand{\MR}{\relax\ifhmode\unskip\space\fi MR }
\providecommand{\MRhref}[2]{%
  \href{http://www.ams.org/mathscinet-getitem?mr=#1}{#2}
}
\providecommand{\href}[2]{#2}
\begin{thebibliography}{BLW19}

\bibitem[BLW19]{BalLidWon21:LagCobHFK}
John~A. Baldwin, Tye Lidman, and C.-M.~Michael Wong, \emph{Lagrangian
  cobordisms and {L}egendrian invariants in knot {F}loer homology}, {M}ichigan
  {M}ath.\ {J}., to appear, version 1, 2019,
  \href{http://arxiv.org/abs/1907.09654}{\texttt{arXiv:1907.09654}}.

\bibitem[BS18]{BalSiv18:KHMLeg}
John~A. Baldwin and Steven Sivek, \emph{Invariants of {L}egendrian and
  transverse knots in monopole knot homology}, 2018, 959--1000. \MR{3917725}

\bibitem[BST15]{BouSabTra15:LagCobGF}
Fr\'{e}d\'{e}ric Bourgeois, Joshua~M. Sabloff, and Lisa Traynor,
  \emph{Lagrangian cobordisms via generating families: construction and
  geography}, Algebr. Geom. Topol. \textbf{15} (2015), no.~4, 2439--2477.
  \MR{3402346}

\bibitem[BTY13]{BorTraYan13:MinLeg}
Bianca Boranda, Lisa Traynor, and Shuning Yan, \emph{The surgery unknotting
  number of {L}egendrian links}, Involve \textbf{6} (2013), no.~3, 273--299.
  \MR{3101761}

\bibitem[Cha10]{Cha10:LagConc}
Baptiste Chantraine, \emph{Lagrangian concordance of {L}egendrian knots},
  Algebr. Geom. Topol. \textbf{10} (2010), no.~1, 63--85. \MR{2580429}

\bibitem[CNS16]{CorNgSiv16:LagConcObstructions}
Christopher Cornwell, Lenhard Ng, and Steven Sivek, \emph{Obstructions to
  {L}agrangian concordance}, Algebr. Geom. Topol. \textbf{16} (2016), no.~2,
  797--824. \MR{3493408}

\bibitem[Dim16]{Dim16:LegAmbSurg}
Georgios Dimitroglou~Rizell, \emph{Legendrian ambient surgery and {L}egendrian
  contact homology}, J. Symplectic Geom. \textbf{14} (2016), no.~3, 811--901.
  \MR{3548486}

\bibitem[EF98]{EliFra98:UnknotTransSimple}
Yakov Eliashberg and Maia Fraser, \emph{Classification of topologically trivial
  {L}egendrian knots}, Geometry, topology, and dynamics ({M}ontreal, {PQ},
  1995), CRM Proc. Lecture Notes, vol.~15, Amer. Math. Soc., Providence, RI,
  1998, pp.~17--51. \MR{1619122}

\bibitem[EG98]{EliGro98:LagrIntThy}
Yasha Eliashberg and Misha Gromov, \emph{Lagrangian intersection theory:
  finite-dimensional approach}, Geometry of differential equations, Amer. Math.
  Soc. Transl. Ser. 2, vol. 186, Amer. Math. Soc., Providence, RI, 1998,
  pp.~27--118. \MR{1732407}

\bibitem[EHK16]{EkhHonKal16:LagCob}
Tobias Ekholm, Ko~Honda, and Tam\'{a}s K\'{a}lm\'{a}n, \emph{Legendrian knots
  and exact {L}agrangian cobordisms}, J. Eur. Math. Soc. (JEMS) \textbf{18}
  (2016), no.~11, 2627--2689. \MR{3562353}

\bibitem[GJ19]{GolJuh19:LOSSConc}
Marco Golla and Andr\'{a}s Juh\'{a}sz, \emph{Functoriality of the {EH} class
  and the {LOSS} invariant under {L}agrangian concordances}, 2019, 3683--3699.
  \MR{4045364}

\bibitem[Hon00]{Hon00:ClassificationI}
Ko~Honda, \emph{On the classification of tight contact structures. {I}}, Geom.
  Topol. \textbf{4} (2000), 309--368. \MR{1786111}

\bibitem[Kan98]{Kan98:NonExactness}
Yutaka Kanda, \emph{On the {T}hurston-{B}ennequin invariant of {L}egendrian
  knots and nonexactness of {B}ennequin's inequality}, Invent. Math.
  \textbf{133} (1998), no.~2, 227--242. \MR{1632790}

\bibitem[Laz20]{Laz20:MaxContSymp}
Oleg Lazarev, \emph{Maximal contact and symplectic structures}, J. Topol.
  \textbf{13} (2020), no.~3, 1058--1083. \MR{4100126}

\bibitem[LM]{LivMoo:KnotInfo}
Charles Livingston and Allison~H. Moore, \emph{Knot{I}nfo: {T}able of {K}not
  {I}nvariants}, available at \url{https://knotinfo.math.indiana.edu}, accessed
  on Apr 28, 2021.

\bibitem[OP12]{ODoPav12:LegGraphs}
Danielle O'Donnol and Elena Pavelescu, \emph{On {L}egendrian graphs}, Algebr.
  Geom. Topol. \textbf{12} (2012), no.~3, 1273--1299. \MR{2966686}

\bibitem[Pan17]{Pan17:LagCobAug}
Yu~Pan, \emph{The augmentation category map induced by exact {L}agrangian
  cobordisms}, Algebr. Geom. Topol. \textbf{17} (2017), no.~3, 1813--1870.
  \MR{3677941}

\bibitem[ST13]{SabTra13:LagCobObstructions}
Joshua~M. Sabloff and Lisa Traynor, \emph{Obstructions to {L}agrangian
  cobordisms between {L}egendrians via generating families}, Algebr. Geom.
  Topol. \textbf{13} (2013), no.~5, 2733--2797. \MR{3116302}

\end{thebibliography}
